\documentclass[a4paper,10.95pt]{amsart}
\usepackage[utf8]{inputenc}
\usepackage{amsmath}
\usepackage{amssymb}
\usepackage{graphicx, overpic}
\usepackage{caption}
\usepackage{subcaption}
\usepackage{enumitem}
\usepackage{mathtools}
\usepackage{color}

\usepackage{hyperref}
\usepackage[hmargin=1.2in,vmargin=1.2 in]{geometry}
\usepackage[capitalize,noabbrev]{cleveref} 

\newcommand{\itemref}[2]{\ref{#1}~(\ref{#2})}

\newtheorem{thm}{Theorem}[section]
\newtheorem{prop}[thm]{Proposition}
\newtheorem{lemma}[thm]{Lemma}
\newtheorem{cor}[thm]{Corollary}

\theoremstyle{definition}
\newtheorem{defn}[thm]{Definition}
\newtheorem{remark}[thm]{Remark}

\newtheorem{example}[thm]{Example}

\newtheorem{ques}[thm]{Question}

\newtheorem*{thm*}{Theorem}
\newtheorem*{prop*}{Proposition}
\newtheorem*{cor*}{Corollary}
\newcommand{\wallmap}{f}
\newcommand{\field}[1]{\mathbb{#1}}
\newcommand{\Z}{\field{Z}}

\newcommand{\R}{\field{R}}

\newcommand{\N}{\field{N}}

\newcommand{\embed}{\hookrightarrow}

\newcommand{\relMorse}[2]{(\partial_M{#2},{#1})}
\newcommand{\Graph}{\Delta}

\newcommand{\pathlist}{\bar}
\newcommand{\open}[1]{\breve{#1} }

\newcommand{\morse}{\partial_M}

\newcommand{\graphclass}{\mathcal C}

\newcommand{\graphzero}{\Delta_1}
\newcommand{\graphone}{\Delta_2}

\newcommand{\B}{\mathcal B}

\newcommand{\Tblock}{\mathcal T}
 
\newcommand{\base}{p}

\newcommand{\Space}{\Sigma}

\newcommand{\vis}{\partial}

\newcommand{\rand}[1]{\partial #1} 

\newcommand{\as}[1]{#1(\infty)}

\newcommand{\typeB}{C}
\newcommand{\typeAf}{A}
\newcommand{\typeAinf}{B}

\newcommand{\wallepsilon}{d_{\mathcal{W}}}

\newcommand{\W}[1]{W_{#1}}
\newcommand{\inv}[1]{#1^{-1}}
\newcommand{\davislambda}{\Sigma_{\Delta}}
\newcommand{\Cayley}[2]{Cay(#1,#2)}
\newcommand{\daviszero}{\Sigma_{\Delta_1}}
\newcommand{\davisone}{\Sigma_{\Delta_2}}

\definecolor{amethyst}{rgb}{0.6, 0.4, 0.8}

\newcommand{\hide}[1]{}
\title{Right-angled Coxeter groups with totally disconnected Morse boundaries}
\author {Annette Karrer}

\begin{document}

\begin{abstract}
		This paper introduces a new class of right-angled Coxeter groups with totally disconnected Morse boundaries.	We construct this class recursively by examining how the Morse boundary of a right-angled Coxeter group changes if we glue a graph to its defining graph. More generally, we present a method to construct amalgamated free products of CAT(0) groups with totally disconnected Morse boundaries that act geometrically on CAT(0) spaces that have a treelike block decomposition.
\end{abstract}

\maketitle

\vskip.2in



 \section{Introduction}
This paper presents new examples of right-angled Coxeter groups that have totally disconnected Morse boundaries. These examples arise from a more general construction of CAT(0) spaces with treelike block decompositions that have totally disconnected Morse boundaries. 
 \subsection{Motivation} 

 \begin{figure}[h]
 	\includegraphics{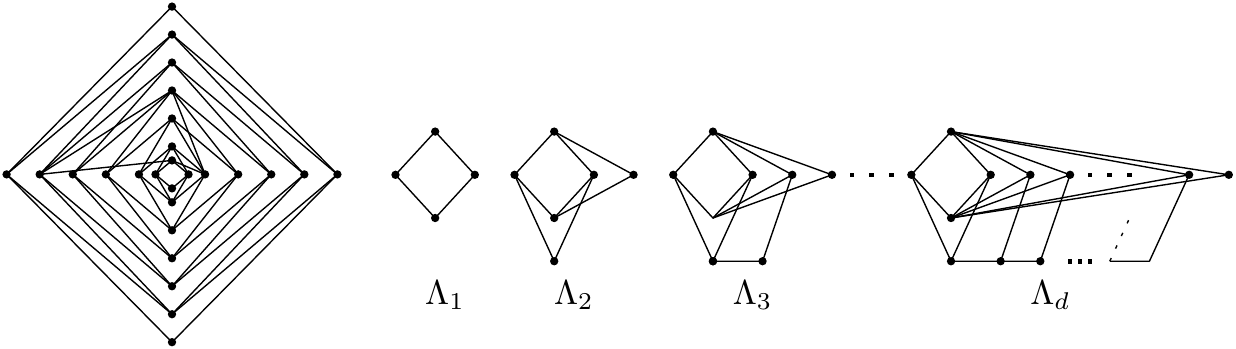}
 			\caption{Graphs that were studied in \cite[Ex. 7.7]{Tran_hierarchically} and \cite[Fig. 5.1]{Dani_diver}. We show that they correspond to RACGs with totally disconnected Morse boundaries. }
 	\label{fig}
 \end{figure}
 The \textit{Morse boundary} $\morse \Sigma$ of a proper geodesic metric space $\Sigma$ is a quasi-isometry invariant defined by Cordes \cite{Cordes_properMorse}. It generalizes the \textit{contracting boundary} introduced by Charney--Sultan \cite{CharSul} in the CAT(0) case. If $\Sigma$ is a proper, geodesic hyperbolic space, its Morse boundary $\morse \Sigma$ coincides with the Gromov boundary. In general, $\morse \Sigma$ is a topological space consisting of equivalence classes of Morse geodesic rays, i.e. geodesic rays that behave similar to geodesic rays in hyperbolic spaces. 
 
 The \textit{Morse boundary} $\morse G$ of a finitely generated group $G$ is the Morse boundary of a proper geodesic metric space on which $G$ acts geometrically, ie. properly and cocompactly by isometries. 
If every geodesic ray bounds a half-flat, e.g. in higher-rank lattices, then $\morse G$ is empty. However, there is a large class of non-hyperbolic finitely generated groups with non-empty Morse boundaries.

 Interesting examples arise among right-angled Coxeter groups (RACGs) and right-angled Artin groups (RAAGs).
 Each such group is defined by a finite, simplicial graph, its \textit{defining graph} and acts geometrically on an associated CAT(0) cube complex. 
 Charney--Cordes--Sisto  \cite{Artingroups} showed:
 \begin{thm}[Charney--Cordes--Sisto] \label{Artingroups}
The Morse boundary of every RAAG is totally disconnected. It is empty, a Cantor space, an $\omega$-Cantor space or consists of two points. 
\end{thm}   

If a RACG has totally disconnected Morse boundary, its Morse boundary is also homeomorphic to one of the spaces listed in the theorem above by Theorem 1.4 in \cite{Artingroups} (see \cref{subsec1spaces}). But in contrast to RAAGs, it is often difficult to determine whether a RACG has totally disconnected Morse boundary or not as many different topological spaces arise as Morse boundaries of RACGs.

 \subsection{RACGs with totally disconnected Morse boundaries}
 \label{subsecRACGIntro}
 	The right-angled Coxeter group (RACG) associated to a finite, simplicial graph $\Delta=(V,E)$ is the group 
 	\[W_\Delta =\langle V\mid v^2=1~\forall v \in V, uv=vu ~\forall ~\{u,v\} \in E \rangle.\]

 The group $W_\Delta$ acts geometrically on an associated CAT(0) cube complex $\Sigma_{\Delta}$, its \textit{Davis complex}.
 Hence, the Morse boundary of $W_\Delta$, denoted by $\morse W_\Delta$, is the Morse boundary of $\Sigma_{\Delta}$. 
 For instance, if $\Delta$ is a $4$-cycle, then $\Sigma_{\Delta}$ is isometric to $\R^2$ and has empty Morse boundary. If $\Delta$ is a $5$-cycle, then $\Sigma_{\Delta}$ is quasi-isometric to the hyperbolic plane and $\morse \Sigma_{\Delta}$ is a circle. If we glue a $4$-cycle to a cycle of length at least $5$ so that the $4$-cycle contains a non-adjacent vertex-pair of the  other cycle as in \cref{fig_cycles},
 \begin{figure}[h]
 	\centering
 	\includegraphics{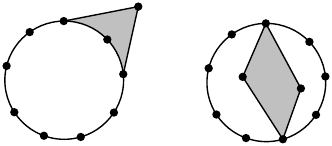} 	
 	\caption{Two cycles with glued $4$-cycles (filled gray).}
 	\label{fig_cycles}
 \end{figure} 
then the corresponding Davis complex has totally disconnected Morse boundary (see \cref{lem:charsultgraph}).
 On the other hand, if a graph $\Delta$ contains an induced cycle $C$ of length at least $5$ without such a glued $4$-cycle, then $\morse \Sigma_\Delta$ contains a circle~\cite[Cor 7.12]{Tran}, \cite[Prop. 4.9]{Genevois}. See also \cite{Behrstock} and \cite[Thm 7.5]{Tran_hierarchically}. 
 Tran conjectured \cite{Tran}[Conj. 1.14] that the non-existence of such a cycle $C$ implies that the associated Davis complex has totally disconnected Morse boundary. This was disproved in \cite{graeber2020surprising}. The problem, which right-angled Coxeter groups have totally disconnected Morse boundaries turns out to be difficult and is still open.
 
  In this paper, we present a new class of right-angled Coxeter groups with totally disconnected Morse boundaries by examining the following question:

 \begin{ques}\label{Q2}
 	Suppose that $\Delta$ is a finite, simplicial graph that can be decomposed into two distinct proper induced subgraphs $\Delta_1$ and $\Delta_2$ with the intersection graph $\Lambda=\Delta_1 \cap \Delta_2$. Are there conditions in terms of $\Delta_1$, $\Delta_2$ and $\Lambda$ implying that $\morse \Sigma_\Delta$ is totally disconnected?
 \end{ques}

\cref{Q2} is inspired by an example of Charney--Sultan~\cite[Sec. 4.2]{CharSul}: 
Let $\bar \Graph$ be the graph in \cref{fig:exCharneySultan}.	 
Charney--Sultan show that $\Sigma_{\bar\Delta}$ has totally disconnected Morse boundary. For the proof, they decompose $\bar\Graph$ into two induced subgraphs $\bar{\Delta}_1$ and $\bar{\Delta}_2$ pictured in~\cref{fig:exCharneySultan}.\begin{figure}[h]
	\centering
	\includegraphics{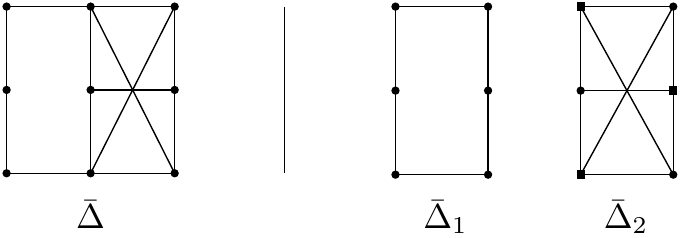}	
\caption{Left: the defining graph of a RACG studied of Charney--Sultan~\cite[Sec. 4.2]{CharSul}. Right: decomposition of the graph.}
	\label{fig:exCharneySultan}
\end{figure} 
	
 Since $\bar{\Delta}_1$ and $\bar{\Delta}_2$ are induced subgraphs of $\bar \Delta$, their corresponding Davis complexes $\Sigma_{\bar{\Delta}_1}$ and $\Sigma_{\bar{\Delta}_2}$ are isometrically embedded in $\Sigma_{\bar \Delta}$. Contrary to the case of visual boundaries, this does not imply that the Morse boundaries $\morse \Sigma_{\bar{\Delta}_1}$ and $\morse \Sigma_{\bar{\Delta}_2}$ are topologically embedded in $\morse \Sigma_{\bar \Delta}$.
\begin{defn}
	Let $\Sigma$ be a proper geodesic metric space and $B \subseteq \Sigma$. 
	We denote by $\relMorse{\Sigma}{B}$ the \textit{relative Morse boundary of $B$ in $\Sigma$}, i.e. the subset of $\morse \Sigma$ that consists of all equivalence classes of geodesic rays in $B$ that are Morse in the ambient space $\Sigma$. 
\end{defn} 

 For instance, if $\Sigma = \R^2$ and $B$ is the $x$-axis, then $\morse B= \{- \infty, + \infty\}$ but $\relMorse{\Sigma}{B} = \emptyset$. 
 If we endow $\relMorse{\Sigma}{B}$ with the subspace topology of $\morse B$ and $\morse\Sigma$, we  obtain two topological spaces that might be distinct (see \cref{exNotcontinuous}). If $B$ is closed and convex, the second topology is finer than the first one (see \cref{blocklemma}). 
 Charney--Sultan use this observation implicitly. They show  in \cite[Sec. 4.2, p. 114-115]{CharSul} that the relative Morse boundaries $\relMorse{\Sigma_{\bar \Delta}}{\Sigma_{\bar{\Delta}_1}}$ and $\relMorse{\Sigma_{\bar \Delta}}{\Sigma_{\bar{\Delta}_2}}$ endowed with the subspace topology of $\morse\Sigma_{\bar{\Delta}_1}$ and $\morse\Sigma_{\bar{\Delta}_2}$ are totally disconnected and conclude that $\morse \Sigma_\Delta$ is totally disconnected. An essential ingredient of their proof is that $\morse \Sigma_{\Delta_2} = \emptyset$.

We generalize this approach using the keyobservation that the intersection graph $\bar{\Delta}_1 \cap \bar{\Delta}_2$ lies in a subgraph of $\bar\Delta$ that corresponds to a RACG with empty Morse boundary (namely $\bar \Delta_2$). RACGs with empty Morse boundary can be characterized in terms of the following definitions.

\begin{defn}
	A graph is a \textit{clique} if every pair of vertices is linked by an edge.
 A graph $\Delta$ is a \textit{join} of two vertex disjoint graphs $\Graph_1$ and $\Graph_2$ if $\Delta$ is obtained from $\Delta_1$ and $\Delta_2$ by linking each vertex of $\Delta_1$ with each vertex of $\Delta_2$. If neither $\Graph_1$ nor $\Graph_2$ is a clique, then $\Delta$ is a \textit{non-trivial join}. 
\end{defn}
For instance, the graph $\bar{\Delta}_2$ is a non-trivial join of two graphs each consisting of three vertices. 
Corollary B in \cite{Sageev_Caprace} implies 
\begin{lemma}[Caprace--Sageev]
	\label{lem:con_empty}
	A RACG has empty Morse boundary if and only if its defining graph is a clique or a non-trivial join.
\end{lemma}

 We are now able to formulate the main result of this paper. 
 \begin{thm}
 	\label{thm1}
 Suppose that $\Delta$ is a finite, simplicial graph that can be decomposed into two distinct proper induced subgraphs $\Delta_1$ and $\Delta_2$ with the intersection graph $\Lambda=\Delta_1 \cap \Delta_2$.
 Suppose that $\Lambda$ is a clique or contained in a non-trivial join of two induced subgraphs of $\Graph$.
 	Then every connected component of $\morse \Sigma_\Delta$ is either	
 	\begin{enumerate}
 		\item a single point; or
 		\item homeomorphic to a connected component of $\relMorse{\Sigma_{\Delta}}{\Sigma_{\Delta_i}}$ endowed with the subspace topology of $\morse \Sigma_{\Delta}$ where $i \in \{1,2\}$.
 	\end{enumerate} 	
 \end{thm}
Our study of relative Morse boundaries in \cref{corOpenset} in \cref{subsecMorse3} implies
 \begin{cor}
 	\label{Cor1}
 Suppose that the assumptions of \cref{thm1} are satisfied. 
 If $\relMorse{\Sigma_{\Delta}}{\Sigma_{\Delta_1}}$ and $\relMorse{\Sigma_{\Delta}}{\Sigma_{\Delta_2}}$ equipped with the subspace topology of $\morse \Sigma_{\Delta_1}$ and $\morse \Sigma_{\Delta_2}$ are totally disconnected then 
 $\morse \Sigma_\Delta$ is totally disconnected. 
 \end{cor}
In \cref{def:graphclass}, we will define a large class $\mathcal C$ of graphs, that can be built iteratively from pieces to which \cref{Cor1} can be applied.
\begin{cor}
	\label{cor:C-class}
	If $\Delta \in \mathcal C$, then $\morse W_\Delta$ is totally disconnected. 
\end{cor}
The class $\mathcal C$ is much larger than the class $\mathcal{CFS}_0$ defined in \cref{def:CFS0} below for which 
 \cref{cor:C-class} was established by Nguyen--Tran \cite{Tran_coars_geom}. 
   For instance, the graphs in \cref{fig} are contained in $\mathcal C \setminus\mathcal{CFS}_0$. 
The left graph was studied by Russell--Spriano--Tran~\cite[Ex. 7.7]{Tran_hierarchically}. They asked whether the associated RACG has totally disconnected Morse boundary or not. The other graphs in  \cref{fig} correspond to RACGs with polynomial divergence of arbitrarily high degree~\cite[Sec. 5]{Dani_diver} (see \cref{divlem}). 
In contrast, all graphs in $\mathcal{CFS}_0$ have quadratic divergence.
 \subsection{CAT(0) spaces with a treelike block decomposition that have totally disconnected Morse boundaries}

 Our results concerning RACGs follow from a more general theorem concerning groups acting geometrically on \textit{CAT(0) spaces with treelike block decompositions}. Such spaces were studied in \cite{CrokeKleiner,BenzviFlats,BENZVI,Mooney} since they arise naturally as spaces on which interesting examples of amalgamated free products of CAT(0) groups act geometrically. We will give a precise definition 
 in \cref{def:blockdec} below. For this introduction, it suffices to know that a \textit{block decomposition} $\mathcal B$ of a CAT(0) space $\Sigma$ is a collection of convex, closed subsets of $\Sigma$, called \textit{blocks}, whose union covers $\Sigma$. The non-trivial intersection of a pair of blocks is called a \textit{wall}. The block decomposition is \textit{treelike} if the blocks intersect each other so that we obtain a simplicial tree if we add a vertex for every block and an edge for every pair of blocks that intersect non-trivially.

 \begin{thm}
	\label{thm2}
	Let $\Sigma$ be a proper CAT(0) space with treelike block decomposition $\mathcal B$. 
	If no wall in $\Space$ contains a geodesic ray that is Morse in $\Space$, then every connected component of $\morse \Space$ 	is either 
	\begin{enumerate}
		\item a single point; or
		\item homeomorphic to a connected component of $\relMorse{\Space}{B}$, where $B$ is a block in $\mathcal B$ and $\relMorse{\Space}{B}$ is endowed with the subspace topology of $\morse \Sigma$. 
	\end{enumerate} 	
\end{thm}
By means of \cref{corOpenset} in \cref{subsecMorse3} we conclude
\begin{cor}
	\label{cor:firstgeneralization}
	Let $\Sigma$ be a proper CAT(0) space with a treelike block decomposition. If no wall contains a geodesic ray that is Morse in $\Space$ and $\relMorse{\Sigma}{B}$  equipped with the subspace topology of $\morse B$ is totally disconnected for every block $B$, then $\morse\Space$ is totally disconnected. 
\end{cor}

 
  \cref{thm1} is a special case of \cref{thm2}. Indeed, in \cref{Prop} we will show that a Davis complex of a RACG as in \cref{thm1}  admits a treelike block decomposition whose blocks are isometric to $\Sigma_{\graphzero}$ or $\Sigma_{\graphone}$ and  whose walls are isometric to $\Sigma_{\Lambda}$. Because of  \cref{lem:con_empty}, this decomposition satisfies the the conditions of   \cref{thm2}.

\subsection{Beyond RACGs} 
 \cref{thm2} has many applications beyond RACGs. It can be applied to a class of RAAGs partially reproving \cref{Artingroups} (see \cref{RAAGThm}), to surface amalgams and to spaces arising from the equivariant gluing theorem of Bridson--Haefliger~\cite[Thm II.11.18]{BH}. We will finish this paper with a few concrete examples that were studied by Ben-Zvi \cite{BenzviFlats}.
 
 \subsection{Organization of the paper}
\cref{sec:treelike} concerns treelike block decompositions of CAT(0) spaces. In \cref{sec:visprop}, we will prove a cutset property for visual boundaries of CAT(0) spaces with a fixed treelike block decomposition. In \cref{key_Morse}, we will transfer this property to the Morse boundary and study two further keyproperties of Morse boundaries. In \cref{subsec:types}, we will use these three keyproperties to prove \cref{thm2}. In \cref{sec:mainproofs}, we will apply our insights to RACGs. 
 Finally, we close this paper with applications beyond RACGs in \cref{examoplesection}.

\subsection*{Acknowledgment}
This paper is part of my dissertation at the Karlsruhe Institute of Technology and I thank my supervisor Petra Schwer for accompanying the process of my
dissertation. I am grateful for the support of my second supervisor Tobias Hartnick, especially for his help with this paper. I would like to thank Matthew Cordes, Nir Lazarovich, Ivan Levcovitz, Michah Sageev and Emily Stark for everything I learned from them and their helpful advice on this paper and beyond during and after my stay at the Technion in Haifa in 2018. I thank Pallavi Dani, Thomas Ng and my PhD collegues Marius Graeber, Leonid Grau, Julia Heller and Kevin Klinge for helpful discussions. Also, I thank Ruth Charney, Jacob Russell and Hung Cong Tran for their comments and Elia Fioravanti for finding an error in the earlier drafts of this paper. Finally, I acknowledge funding of the Deutsche Forschungsgemeinschaft (DFG 281869850, RTG 2229), the Karlsruhe House of Young Scientists (KHYS), and the Israel Science Foundation (grant no. 1562/19).


 \section{CAT(0) spaces with treelike block decompositions}
 \label{sec:treelike}
 In \cref{sec1}, we will fix notation. \cref{sec:blockdeck} concerns basic properties of CAT(0) spaces with treelike block decompositions. \cref{sec:itin} is about itineraries of geodesic rays in such spaces. 
 \subsection{Notation concerning simplicial graphs}
 \label{sec1}
 For the background of graphs, see \cite{West}.
 For us, a \textit{simplicial graph} $\Graph=(V(\Delta), E(\Delta))$ consists of a set $V(\Delta)$ and set $ E(\Delta)$ of $2$-element subsets of $ V(\Delta)$.
 The elements of $V$ are called \textit{vertices} and the elements of $ E$ are called \textit{edges}. If $\Graph_1$ and $\Graph_2$ are two graphs, then $\Graph_1 \cup \Graph_2$ is the graph whose vertex set is $V(\Graph_1) \cup V(\Graph_2)$ and whose edge set is the set $E(\Graph_1)\cup E(\Graph_2)$. Analogously, $\Graph_1 \cap \Graph_2$ denotes the graph whose vertex set is $V(\Graph_1) \cap V(\Graph_2)$ and whose edge set is the set $E(\Graph_1)\cap E(\Graph_2)$. 
 A \textit{subgraph} $\Graph'$ of a graph $\Graph$ is a graph whose vertex set is contained in $V(\Graph)$ and whose edge set is contained in $E(\Graph)$. The subgraph $\Graph'$ is a \textit{proper} subgraph if it does not coincide with $\Graph$. A graph $\Graph'$ is an \textit{induced subgraph} of a graph $\Graph$ if every edge $e \in E(\Graph)$ whose endvertices are contained in $V(\Delta')$ is contained in $E(\Graph')$. We say in this case that $\Graph'$ is \textit{spanned} by the vertex set $ V'$. 
 
 Two vertices are \textit{adjacent} if they are contained in an edge. The \textit{degree} of a vertex $v$ is the number of vertices that are adjacent to $v$.
Let $v_1,\dots,v_n \in V(\Delta)$ and $e_1,\dots,e_{n-1} \in E(\Delta)$. 
The list $P=(v_1,e_1,v_2,e_2,\dots e_{n-1}v_n)$ is a \textit{finite path} of length $n-1$, if $v_i \in e_i$ for all $i \in \{1,\dots,n-1\}$ and $v_n \in e_{n-1}$. In this case, $v_1$ and $v_n$ are \textit{linked} by a finite path. If $v_1 = v_n$, $P$ is a \textit{closed path}. 
Let $(v_i)_{i \in \N}$ and $(e_i)_{i \in \N}$ be two sequences of vertices and edges in $V(\Delta)$ and $E(\Delta)$ respectively. 
The infinite list $(v_1,e_1,v_2,e_2,\dots e_{n-1}v_n\dots)$ is an \textit{infinite path} if $v_i \in e_i$ for all $i \in \N$. 
Let $(v_i)_{i \in \Z}$ and $(e_i)_{i \in \Z}$ be two sequences of vertices and edges in $V(\Delta)$ and $E(\Delta)$ respectively. 
The bi-infinite list $(\dots, v_{-1},e_{-1},v_0,e_0,v_1,e_1,v_2,e_2,\dots e_{n-1}v_n\dots)$ is a \textit{bi-infinite path} if $v_i \in e_i$ for all $i \in \Z$. 
If we speak of a \textit{path}, we mean a finite, infinite or bi-infinite path. 

 A path $P$ is \textit{geodesic}, if each vertex occurs at most once in $P$. A path $P$ is a \textit{subpath} of a path $P'$ if $P'$ has two vertices $v_1$ and $v_2$ so that $P$ is obtained from $P'$ by removing all vertices and edges that occur  before the vertex $v_1$ or after the vertex $v_2$ in $P'$. The \textit{underlying graph of a path $P$}, denoted by $\pathlist P$, is the graph whose vertex set consists of all vertices in $P$ and whose edge set consists of all edges in $P$.  If $P$ is a geodesic path, each vertex in $\pathlist P$ has degree at most two. 

 A graph is \textit{connected} if every two vertices are linked by a finite path. 
 A \textit{cycle} is a graph with an equal number of vertices and edges whose vertices can be placed around a cycle so that two vertices are adjacent if and only if they appear consecutively along the cycle. A graph is a (simplicial) \textit{tree} if it is connected and does not contain a cycle. An important property of trees is that every geodesic path linking two vertices is unique. 
 
 \subsection{Definitions and basic properties}
\label{sec:blockdeck}

In this subsection, we study CAT(0) spaces that have a \textit{treelike block decomposition}. 
The following considerations are variants of definitions and lemmas in \cite{CrokeKleiner,BenzviFlats,BENZVI,Mooney}. For the background about CAT(0) spaces, see~\cite[Ch. II]{BH}.
\begin{defn}
	\label{def:blockdec}
	Let $\Sigma$ be a CAT(0) space. A collection $\mathcal B$ of closed convex subsets of $\Sigma$ is a \textit {block decomposition} of $\Sigma$ if it satisfies the covering condition 
		\[\Sigma = \bigcup_{B \in \B}B.\] 
	In this case, the elements of $\B$ are called \textit{blocks}. A block decomposition is \textit{non-trivial} if there are at least two blocks. If $\mathcal B$ is a non-trivial block decomposition, the elements of the set \[\mathcal W\coloneqq \{B_1 \cap B_2 \mid B_1, B_2 \in \B \text{ such that } B_1 \neq B_2\}\setminus \{\emptyset\}.\]
 are called \textit{walls}.
 
Let $d: \Sigma \times \Sigma \to \R_{\ge 0}$ be the metric of $\Sigma$. We define the distance of two walls $W_1$, $W_2 \in \mathcal W$ by \[d(W_1,W_2) \coloneqq \inf_{x_1 \in W_1, x_2 \in W_2}{d(x_1,x_2)}.\]
 \end{defn}

\begin{defn}
	The \textit{adjacency graph} of a block decomposition $\mathcal B$ is the simplicial graph whose vertex set is $\mathcal B$ and whose set of edges consists of all pairs of blocks with non-empty intersection.
\end{defn}
Recall that a simplicial tree is a connected simplicial graph that does not contain a cycle.
\begin{defn}\label{deftreelike}
	A block decomposition is called \textit{treelike} if
	\begin{enumerate}
		\item the adjacency graph is a simplicial tree; and if
		\item there exists $\wallepsilon > 0$ so that for all $W_1, W_2 \in \mathcal W$, $W_1 \neq W_2$, we have $d(W_1, W_2)\ge \wallepsilon$ (separating property).\label{epqilon-cond} 
	\end{enumerate} 
\end{defn}
	Let $\Sigma$ be a complete CAT(0) space with treelike block decomposition $\mathcal B$ and $\Tblock$ the corresponding adjacency graph.
The following two properties are important for us.
\begin{lemma}
	If $\Sigma$ is proper, then no ball of finite radius in $\Sigma$ is intersected by infinitely many walls. \label{ball}
\end{lemma}
\begin{proof}
		Let $B$ be a ball of finite radius. As $\Sigma$ is proper, $B$ is compact and we are able to cover $B$ by finitely many balls of radius $\frac{\wallepsilon}{4}$. By the separating property \itemref{deftreelike}{epqilon-cond}, each ball of radius $\frac{\wallepsilon}{4}$ is intersected by at most one wall. Hence, the number of walls intersecting $B$ is less or equal to the number of the balls that are used for covering $B$.\end{proof}

 \begin{lemma}
	Let $T$ be a (possibly infinite) subtree of $\Tblock$ with vertex set $V$. 
	Then the set $\bigcup_{B \in V(T)}{B}$ is closed and convex.\label{convex}
\end{lemma}

	 We need the following lemmas for proving \cref{convex}.

\begin{lemma}
 The intersection of more than two blocks is empty. In particular, every two distinct walls are disjoint and every point in $\Sigma$ is either contaied in exactly one wall or in exactly one block but not both. \label{1}
\label{2}
\label{welldef}
\end{lemma}
\begin{proof}
 If $B_1, \dots,B_k \in \mathcal B$ with $B_1 \cap B_2\cap \dots \cap B_k \neq \emptyset$, then $\mathcal T$ contains a clique on $k$ vertices as subgraph. Since $\mathcal T$ is a tree, $k \le 2$.
 
 We show that distinct walls are disjoint: Indeed, let $W_1 = B_1\cap B_2$ and $W_2= B_3\cap B_4$ be two distinct walls where $B_1$, $B_2$, $B_3$, $B_4 \in \mathcal B$. As $W_1$ and $W_2$ are distinct, there exists $i \in \{1,2\}$ such that $B_i \notin\{B_3, B_4\}$. 
 By definition of walls, $B_1 \ne B_2$ and $B_3 \ne B_4$. As the intersection of more than two blocks is empty and $B_i \cap B_j \neq \emptyset$ where $j \in \{1,2\}$, $j \ne i$, the block $B_i$ does neither intersect $B_3$ nor $B_4$. Thus, $W_1 = B_1 \cap B_2$ does not intersect $W_2 = B_3 \cap B_4$. 

	It remains to show that every $x \in \Sigma$ is either contained in exactly one wall or in exactly one block but not both: Let $x \in \Sigma$. Suppose that $x$ is not contained in exactly one block. Then $x$ is contained in at least two blocks. The intersection of more than two blocks is empty. Thus, $x$ is contained in exactly two blocks. This means, that $x$ is contained in a wall. As distinct walls are disjoint, there exists exactly one wall containing $x$.\end{proof}

\begin{lemma}
	\label{wallconvex}	Each wall is closed and convex and 
 the map \begin{align*}
		\wallmap:~&E(\Tblock) \to W \\
		&\{B_1,B_2\} \mapsto B_1 \cap B_2
	\end{align*}
	is a bijection. \label{bijetion}
\end{lemma}

\begin{proof}
The intersection of two closed convex sets is closed and convex. Every wall is the intersection of two blocks and by definition, each block is closed and convex. Thus, each wall is closed and convex.
	
	Next, we show that $\wallmap$ is surjective. Let $W$ be a wall. Then there are two blocks $B_1$, $B_2$, $B_1 \ne B_2$ such that $W = B_1 \cap B_2$ and $\mathcal T$ contains the edge $\{B_1,B_2\}$. By definition, $f(\{B_1,B_2\})= B_1 \cap B_2 = W$. Thus, $\wallmap$ is surjective. 
	
	It remains to prove that $\wallmap$ is injective. Let $e_1=\{B_1,B_2\}$ and $e_2=\{B_3,B_4\}$ be two edges in $\mathcal T$ such that $f(e_1) = f(e_2)$. Then $B_1 \cap B_2 \neq \emptyset$ and $B_3\cap B_4 \neq \emptyset$ and $B_1 \cap B_2= B_3 \cap B_4$. 
	By \cref{1}, each point in $\Sigma$ lies in at most two blocks. Hence, $\{B_1,B_2\}$ =$\{B_3,B_4\}$, i.e. $e_1=e_2$.\end{proof}

For avoiding the use of the term "boundary" in two different meanings, we define the \textit{topological frontier} of a set $S$ to be the closure of $S$ minus the interior of $S$. If $x$ is a point in $\Sigma$ and $\epsilon>0$, we denote by $U_\epsilon(x)$ the open $\epsilon$-neighborhood about $x$. If $B$ is a block in a block decomposition of a CAT(0) space with wall-set $\mathcal W$, then $\mathcal W_B$ denotes the set $\{W \in \mathcal W \mid W \subseteq B\}$ and $\open{B}\coloneqq B\setminus \mathcal W_B$.
	\begin{lemma}
 For every $B \in \mathcal B$, the set $\open{B}=B\setminus \mathcal W_B$ is open in $\Sigma$. \label{inneropen}
\end{lemma}
\begin{proof}
	 Let $B$ be a block and $x \in \open{B}= B\setminus \mathcal W_B$. Let $\delta \coloneqq \inf_{t \in \mathbb R_{\ge 0}}\{d(x,y) \mid y \in \mathcal W_B\}$. It suffices to show that $\delta>0$, because then $U_{\frac{\delta}{2}}(x) \subseteq \open B$. Assume for a contradiction that $\delta = 0$. Then for each $\epsilon >0$ there exists a wall $W_\epsilon \in \mathcal{W}$ such that $W \cap U_\epsilon(x) \neq \emptyset$. By the separating property \itemref{deftreelike}{epqilon-cond}, $W_\epsilon= W_{\epsilon'}$ for all $\epsilon, \epsilon' \in (0,\wallepsilon)$. Thus, there exists a wall $W$ such that $U_\epsilon(x) \cap W\neq \emptyset$ for all $\epsilon>0$. Hence, $x$ is a limit point of $W$. By \cref{wallconvex}, $W$ is closed. So, $W$ contains all its limits points. In particular, $x \in W$. But then, $x \notin \open{B}$ -- a contradiction to the choice of $x$.\end{proof}
\begin{lemma}\label{infimum2}\label{infimum3}\label{curvebetweenwalls}
	Let $c:[a,b] \to \Sigma$ be a curve connecting two distinct blocks $B_1$ and $B_2$. Let
	 \begin{align*}
	t_0 \coloneqq \inf\{t \in [a,b]\mid \gamma(t)\notin B_1\} \text{ and } t_1 \coloneqq \inf\{t \in [a,b]\mid \gamma(t) \in B_2\}.
	\end{align*} Then $\gamma(t_0) \in \mathcal W_{B_1}$ and $\gamma(t_1) \in \mathcal W_{B_2}$. If $t_0 \notin B_1 \cap B_2$ or $t_2 \notin B_1 \cap B_2$, then there exists $t' \in (t_0, t_1)$ such that $\gamma(t')$ is not contained in a wall. 
\end{lemma}
\begin{proof}
 Suppose that we have already proven that $\gamma(t_0) \in \mathcal W_{B_1}$ and $\gamma(t_1) \in \mathcal W_{B_2}$. By \cref{2}, $\mathcal W_{B_1}\cap \mathcal W_{B_2}=\{B_1\cap B_2\}$ where $B_1 \cap B_2$ might be the empty set. If $t_0 \notin B_1 \cap B_2$ or $t_2 \notin B_1 \cap B_2$, then $\gamma(t_0)$ and $\gamma(t_1)$ are contained in two distinct walls. Then $\gamma([t_0,t_1])$ is a curve connecting two distinct walls. By the separating property \itemref{deftreelike}{epqilon-cond}, the distance of two distinct walls is at least $\wallepsilon$. Thus, $\gamma((t_0,t_1))$ contains a point $x$ that is not contained in a wall. 
 
It remains to prove that $\gamma(t_0)$ lies in $\mathcal W_{B_1}$ and that $\gamma(t_1)$ lies in $\mathcal W_{B_2}$. We focus on proving that $\gamma(t_0)$ lies in $\mathcal W_{B_1}$. Therefore, we observe that the topological frontier of $B_1$ is contained in $\mathcal W_{B_1}$. Indeed, Let $x$ be a point in the topological frontier of $B_1$. As $B_1$ is closed, $x \in B_1$. As the set $\open{B_1}$ is open by \cref{inneropen}, $\open{B_1}$ is contained in the interior of $B_1$. Thus, $x \notin \open{B_1}$. Hence, $x \in B_1 \setminus \open{B_1}$, i.e. $x \in \mathcal W_{B_1}$ and the topological frontier of $B_1$ is contained in $\mathcal W_{B_1}$. 

Since the topological frontier of $B_1$ is contained in $\mathcal W_{B_1}$, it suffices to prove that $\gamma(t_0)$ is contained in the topological frontier of $B_1$. First, we show that $\gamma(t_0)$ is a limit point of $B_1$. Indeed otherwise, there would exist $\epsilon>0$ such that $U_\epsilon(\gamma(t_0)) \cap B_1 = \emptyset$. As $\gamma$ starts in $B_1 \subseteq \Sigma \setminus U_\epsilon(\gamma(t_0))$, $\gamma$ connects a point outside of $U_\epsilon(\gamma(t_0))$ with $\gamma(t_0)$. Thus,
		there would exist $\epsilon'>0$ such that $\gamma((t_0-\epsilon', t_0]) \subseteq U_\epsilon(\gamma(t_0))$. Then $\gamma((t_0-\epsilon ', t_0]) \cap B_1= \emptyset$ -- a contradiction to the choice of $t_0$.
		
 It remains to prove that $\gamma(t_0)$ is not an interior point of $B_1$.	 
		By the choice of $t_0$, for each $\epsilon \in (0,b-t_0)$ there exists $t \in [t_0, t_0+\epsilon)$ so that $\gamma(t)\notin B_1$. Hence, for each $\epsilon \in (0,b-t_0)$, $U_\epsilon(\gamma(t_0))\nsubseteq B_1$ and $\gamma(t_0)$ is not an interior point of $B_1$.
		This completes the proof that $\gamma(t_0) \in \mathcal W_{B_1}$. A similar argumentation shows that $\gamma(t_1)$ lies in $\mathcal W_{B_2}$.\end{proof}

\begin{lemma}
	If $B_1$, $B_2\in \mathcal B$ such that $B_1 \cap B_2 \neq \emptyset$, then the open $\wallepsilon$-neighborhood of $B_1\cap B_2$ is contained in $B_1\cap B_2 \cup \open{B_1} \cup \open{B_2} $. \label{wallneighborhood}	
\end{lemma}
\begin{proof}
	Let $B_1$, $B_2 \in \mathcal B$ such that $B_1 \cap B_2 \neq \emptyset$. Let $x \in B_1 \cap B_2$. We have to show that $U_{\wallepsilon}(x) \subseteq B_1\cap B_2 \cup \open{B_1} \cup \open{B_2} $. Suppose that this would not be the case. Then there exists a block $B_3 \in \mathcal B \setminus \{B_1,B_2\}$ such that $B_3\cap U \neq \emptyset$. Let $y \in B_3\cap U \neq \emptyset$ and $\gamma$ be the geodesic segment connecting $x$ and $y$. By \cref{infimum3}, $\gamma$ has to pass a wall $W'$ that is contained in $B_3$. As $B_3 \notin \{B_1,B_2\}$, the wall $W'$ does not coincide with $B_1 \cap B_2$. By the separating property \itemref{deftreelike}{epqilon-cond}, $\inf_{y \in W'}{d(x,y)}\ge \wallepsilon$. But this is impossible because $y \in U_{\wallepsilon}(x) $.\end{proof}

\begin{lemma}
		\label{passthrough}
		If $P$ is a path in $\mathcal T$ linking two blocks $B$ and $B'$ and $W$ is a wall corresponding to an edge of $P$, then each curve in $\Sigma$ linking a point in $B$ with a point in $B'$ passes through $W$.
\end{lemma}

\begin{proof}
	We will prove the statement in two steps.
	\begin{enumerate}
		\item[Claim 1] If $B_1$ and $B_2$ are two distinct blocks with non-empty intersection $W$, then $\Sigma \setminus W$ decomposes so that each pair of points
		$x_1 \in B_1 \setminus W$, $x_2 \in B_2 \setminus W$ lie in different connected components of $\Sigma \setminus W$.
		 \label{cuttingcondition}
		\item[Proof:] Let $B_1$ and $B_2$ be two blocks with non-empty intersection $W$. Let $x_1 \in B_1 \setminus W$ and $x_2 \in B_2 \setminus W$. Let $e$ be the edge in $\mathcal T$ corresponding to $W$. If we delete the edge $e$, then $\mathcal T$ decomposes into two trees $T_1$ and $T_2$. Let $T_1$ be the tree containing $B_1$ and $T_2$ be the tree containing $B_2$. Let $O_1 \coloneqq \bigcup_{B \in \mathcal V(T_1)}{B} \setminus W$ and $O_2 \coloneqq\bigcup_{B \in \mathcal V(T_1)}{B} \setminus W$. By \cref{welldef}, each point in $\Sigma$ is contained in exactly one block or exactly one wall. Hence, $O_1 \cap O_2 = \emptyset$ and 
		$\Sigma \setminus W = O_1 \overset{.}{\cup} O_2$ such that $x_1 \in O_1$ and $x_2 \in O_2$. 	If $O_1$ and $O_2$ are open, this implies that $x_1$ and $x_2$ lie in distinct connected components of $\Sigma \setminus W$. Thus, it remains to show that $O_1$ and $O_2$ are open.
		
		By symmetry reasons it suffices to prove that $O_1$ is open. Let $x \in O_1$. If there exists $B \in V(T_1)$ such that $x \in \open B$, then there exists $\epsilon>0$ such that $U_\epsilon(x) \in \open{B}$ as $\open{B}$ is open by \cref{inneropen} and $x$ is an interior point of $O_1$.
		
		It remains to consider the case that there is no block $B$ in $V(T_1)$ such that $x \in B$. In that case, there exists a unique wall $\hat W\neq W$ corresponding to an edge in $T_1$ that contains $x$ by \cref{welldef}. Let $ C_1$, $C_2 \in V(T_1)$ such that $\hat W = C_1 \cap C_2$. By \cref{wallneighborhood}, $U_{\wallepsilon}(x) \subseteq \open{C_1} \cup \open{C_2} \cup (C_1 \cap C_2)$. As $C_1$,$C_2 \in V(T_1)$, the union $\open{C_1} \cup \open{C_2} \cup (C_1 \cap C_2)$ is contained in $O_1$. Hence, $x$ is an interior point of $O_1$. 	
	\end{enumerate}

 Now, we prove the lemma. Let $P$ be a path in $\mathcal T$ linking two blocks and $c:[a,b] \to \Sigma$ a curve linking two points in these two blocks. We have to show that $c$ passes through every wall that corresponds to an edge of $P$. 
Assume for a contradiction that there exists a wall $W$ corresponding to an edge in $P$ such that $c([a,b])\cap W = \emptyset$. Let $B$, $B'\in \mathcal B$ such that $W = B \cap B'$.
 It suffices to find two curves $c_1$ and $c_2$ that don't intersect $W$ so that $c_1$ links a point in $B \setminus W$ with $c(a)$ and $c_2$ links $c(b)$ with a point in $B' \setminus W$. Then the concatenation of $c_1$, $c$ and $c_2$ is a curve in $\Sigma \setminus W$ that connects a point in $B \setminus W$ with a pint in $B' \setminus W$. That contradicts Claim 1. 
 
 It remains to find the curves $c_1$ and $c_2$. Let $B_1 = B$ and $P_1=(B_1,W_1,B_2,W_2,\dots,B_k)$ be the unique geodesic path in $\mathcal T$ connecting $B$ with the first block of $P$. If $P_1$ consists of a single vertex, then $B_k = B$ and $c(a) \in B$. Because we assume that $c$ does not intersect $W$, $c(a) \in B \setminus W$. Then the trivial constant curve with value $c(a)$ is the curve $c_1$ we are looking for. 
 
 It remains to study the case where $P_1$ has length at least one. We assume Without loss of generality that $P_1$ does not contain the vertex $B'$ (otherwise, we switch the roles of $B$ and $B'$). Let $x \in W_1$. As distinct walls are disjoint by \cref{2}, $x \in B \setminus W$. Let $w_1 \coloneqq x$. Let $w_j$ be an arbitrarily chosen point in $W_j$ for $j \in \{2,\dots,k-1\}$. Let $w_k\coloneqq c(a)$. Let	 $\hat c_j$ be the geodesic segment connecting $w_j$ with $w_{j+1}$ for $j \in \{0,\dots,k-1\}$. Finally, let $c_1$ be the concatenation of $\hat c_0$, $\hat c_2$, $\dots$, $\hat c_{k-1}$. 
 
 We have to prove that $c_1$ does not intersect $W$. Assume for a contradiction that there exists $j \in \{0,\dots,k-1\}$ such that $W \cap \hat c_j \neq \emptyset$. By the choice of $\{w_j\}_{j \in \{1,\dots,k\}}$, every pair of two consecutive points $w_j$ and $w_{j+1}$ are contained in $B_{j+1}$ for all $j \in \{0,\dots,k-1\}$. As each block is convex, $\hat c_j \subseteq B_{j+1}$ for all $j \in \{0,\dots,k-1\}$. Thus, if $W \cap \hat c_j \neq \emptyset$, then $W \cap B_{j+1} \neq \emptyset$. Recall that $W =B'\cap B_1$. This implies that $B'\cap B_1 \cap B_{j+1}\neq \emptyset$. As $B'$ is not contained in $P_1$, and as $P_1$ is a geodesic path, the three blocks $B'$, $B_1$ and $B_{j+1}$ are distinct. That contradicts \cref{1}. Thus, $c_1$ does not intersect $W$.
 The curve $c_2$ can be defined analogously.\end{proof}

 \begin{proof}[Proof of \cref{convex}]
 	At first, we show that $M\coloneqq\bigcup_{B \in V(T)}{B}$ is convex. Let $x,y \in \bigcup_{B \in V(T)}{B}$. As $\Sigma$ is CAT(0), there exists a unique geodesic segment $\gamma: [a,b]\to \Sigma$ connecting $x$ and $y$. We have to show that $\gamma([a,b])\subseteq M$. Let $B$ and $B' $ be two blocks in $ V(T)$ containing $x$ and $y$ respectively. \\
 	As $T$ is connected, there exists a geodesic path $P$ connecting $B$ and $B'$. As $T$ is a tree, this path is unique. As $T$ is a subtree of $\mathcal T$, $P$ is a path in $\mathcal T$. We show the statement by induction on the length of $P$. 
 	
 	If $P$ consists of a single vertex, $x$ and $y$ lie in a common block $B = B' \subseteq V(T)$. As each block is convex, $\gamma \subseteq B \subseteq M$. 
 	Now suppose that the claim is true if $P$ has length $k-1$, $k\ge 1$. \\
 	Let $P$ be a path of length $k$. Let $W$ be a wall corresponding to an edge in $P$. By \cref{passthrough}, there exists $t\in [a,b]$ such that $\gamma(t) \in W$. Let $B_1$, $B_2 \in \mathcal B$ such that $W = B_1 \cap B_2$. Let $B'' \in \{B_1,B_2\}$ such that the geodesic path $P'$ connecting $B$ with $B''$ does not contain both $B_1$ and $B_2$. As geodesic paths in trees are unique, $P'$ is a subpath of $P$ that is shorter than $P$. By induction hypothesis, $\gamma([a,t]) \subseteq M$. Analogously, $\gamma([t,b]) \subseteq M$. Thus, $\gamma \subseteq \bigcup_{B \in V(P)}{B} \subseteq M$. 
 	
 	It remains to show that $M$ is closed, i.e. we have to show that $\Sigma \setminus M$ is open. 
 	Let $x \in \Sigma \setminus M$. We have to prove that $x$ is an interior point of $\Sigma \setminus M$. \\
 	First suppose that $x$ is not contained in a wall. Then there exists a unique block $B$ containing $x$ by \cref{welldef}. As $x \in \Sigma \setminus M$, $B$ is not contained in $V(T)$. As $x$ is not contained in a wall, $x \in \open{B}$. By \cref{inneropen}, $\open{B}$ is open. Thus, there exists an open neighborhood $U$ about $x$ that is contained in $\open B$. By definition, $\open B$ consists of points that are not contained in any other block than $B$. Thus, $U \subseteq \open B \subseteq \Sigma \setminus M$ and $x$ is an interior point of $\Sigma \setminus M$. 
 	
 	Suppose that $x$ is contained in a wall $W$. It remains to show that $x$ is an interior point. Let $B_1$, $B_2 \in \mathcal B$ such that $W = B_1 \cap B_2$. Since $x \in \Sigma \setminus M$ and $M=\bigcup_{B \in V(T)}{B}$, the blocks $B_1$ and $B_2$ are not contained in $V(T)$. By \cref{1}, each point in $\Sigma$ is contained in at most two blocks. Thus, no point in $W$ is contained in a block $B \in \mathcal B \setminus\{B_1,B_2\}$ and $W \subseteq \Sigma \setminus M$. By \cref{welldef}, $(\open{B_1}\cup\open{B_2}) \cap B= \emptyset$ for all $B \in V(T)$. Hence, $\open{B_1} \cup \open{B_2}\cup W \subseteq \Sigma \setminus M$.
 	By \cref{wallneighborhood}, $U_{\wallepsilon}(x) \subseteq \open{B_1} \cup \open{B_2}\cup W$. 
 	Thus, $x$ is an interior point of $\Sigma \setminus M$.\end{proof}
 
 \begin{remark}
 It is possible to prove the lemmas in this section for geodesic metric spaces $\Sigma$ without assuming that $\Sigma$ is CAT(0). Furthermore, one can omit the assumption that blocks are convex except for the convexity-statement in \cref{wallconvex} and \cref{convex}.
 \end{remark}

\begin{remark}[Criterion for treelike block decompositions]
\label{criteriontreelike}
If a CAT(0) space satisfies the following three conditions introduced by Mooney \cite{Mooney}, then $\Sigma$ is a CAT(0) space with a treelike block decomposition as in \cref{def:blockdec}. 
\newpage
\begin{enumerate}
	\item $\Sigma= \bigcup_{B \in \mathcal B} B$ (covering condition);
	\item every block has a parity $(+)$ or $(-)$ such that two blocks intersect only if they have opposite parity (parity condition);
	\item there is an $\epsilon >0$ such that two blocks intersect if and only if their $\epsilon$-neighborhoods intersect ($\epsilon$-condition).
\end{enumerate}
 Mooney \cite{Mooney} argues that the adjacency graph of $\mathcal B$ is a tree $\mathcal T$. It remains to show that there exists $\wallepsilon > 0$ such that every two distinct walls have distance at least $\wallepsilon$. Let $\wallepsilon \coloneqq 2\epsilon$ and $W_1$, $W_2 \in \mathcal{W}$, $W_1 \neq W_2$. Let $x \in W_1$, $y \in W_2$. We show that $d(x,y)\ge \wallepsilon = 2\epsilon$. Indeed otherwise, $d(x,y)<2\epsilon$. Since $x$ and $y$ are contained in walls, there are two blocks $B_x^1$ and $B_x^2$ containing $x$ such that $W_1 =B_x^1 \cap B_x^2$ and two blocks $B_y^1$ and $B_y^2$ containing $y$ such that $W_2=B_y^1 \cap B_y^2$. By assumption, $d(x,y)<2\epsilon$. Hence, for each $B_x \in \{B_x^1, B_x^2\}$ and for each $B_y \in \{B_y^1, B_y^2\}$, the intersection $U_\epsilon(B_x)\cap U_\epsilon(B_y)$ is not empty. Thus, the $\epsilon$-condition above implies that $W_1'\coloneqq B_x^1\cap B_y^1 \neq \emptyset$ and that $W_2'\coloneqq B_x^2 \cap B_y^2\neq \emptyset$. Hence, $C\coloneqq(B_x^1, W_1,B_x^2,W_1',B_y^1, W_2, B_y^2, W_2',B_x^2, B_x^2)$ is a closed path in $\mathcal T$. As $W_1 \neq W_2$, the set $\{B_i \mid i \in \{1,2,3,4\}\}$ contains at least three elements. Thus, the underlying graph of $C$ contains a cycle of length at least $3$. This is impossible as $\mathcal T$ is a tree. \end{remark}

\subsection{Itineraries}
\label{sec:itin}

Let $\Sigma$ be a CAT(0) space with treelike block decomposition $\mathcal B$. Let $\mathcal T$ be the adjacency graph of $\mathcal B$. 
From now on, we identify the set of walls $\mathcal W$ with the edges of $\Tblock$. this is possible because of \cref{bijetion}.

\begin{lemma} 
	\label{lem:defitin} Let $\gamma:[a,b] \to \Sigma$ be a geodesic segment that does not start in a wall. Let $I(\gamma)=(B_1,W_1,B_2,\dots,W_{k-1},B_{k})$ be the shortest geodesic path in $\mathcal T$ linking the unique block containing $\gamma(a)$ with one of the (at most two) blocks containing $\gamma(b)$.	
	 Then the times $t_1 \coloneqq a$, $t_i\coloneqq \inf\{t \in [a,b] \mid \gamma(t) \notin B_{i-1}\}$, $i \in \{2,\dots,k\}$ and $t_{k+1} \coloneqq b$ satisfy the following four properties:
	\begin{enumerate}
		\item $\gamma(t_1) \in \open{B_1}$ and $\gamma(t_i) \in W_{i-1} = B_{i-1} \cap B_{i}$ for all $i \in \{2,\dots,k\}$;
		 \item $t_{i+1} - t_{i} \ge \wallepsilon$ for all $i \in \{2,\dots,k-1\}$;
		\item $\gamma([t_{i-1},t_{i}]) \subseteq B_{i-1}$ for all $i \in \{2,\dots,k+1\}$; 
		\item $\gamma([t_{i-1},t_{i}]) \cap \open B_{i-1} \neq \emptyset$ for all $i \in \{2,\dots,k+1\}$.	
	\end{enumerate}	
\end{lemma}

\begin{proof}
	We proof the statement by constructing the path $I(\gamma)$ as follows.
	\begin{itemize}
		\item \textbf{Step 1:} As $\gamma$ does not start in a wall, there exists a unique block $B_1$ such that $\gamma(a) \in \open{B_1}$ by \cref{welldef}. Let $P_1$ be the path consisting of the vertex $B_1$. 
		\item	\textbf{$i+1$-th step:} Let $P=(B_1,W_1,B_2,\dots,W_{i-1},B_{i})$ be the path that we obtain after $i$ steps. 
		\begin{enumerate}
			 	\item[Case 1:] If $\gamma$ ends in $B_i$, $\gamma|_{[t_i,b]}$ is contained in $B_i$ as each block is convex. In this case, let $I(\gamma)\coloneqq P_i$. 
			\item[Case 2:] Suppose that $\gamma$ does not end in $B_i$. Let $t_{i+1} \coloneqq\inf\{t \in [a,b] \mid \gamma(t)\notin B_i\}$. By \cref{infimum2}, $\gamma(t_{i+1})$ is contained in a wall $W_{i+1}$ that intersects $B_i$ non-trivially. We observe that $W_{i+1} \neq W_{i}$ for all $i\ge 2$. Indeed, if $W_i$ and $W_{i+1}$ would coincide, then $\gamma([t_i,t_{i+1}]) \subseteq W_i$ as $W_i$ is convex. Then $\gamma[t_i, t_{i+1}] \subseteq B_{i-1}$ -- a contradiction to the choice of $t_i$. Thus, $W_{i+1} \neq W_{i}$.
			
			 Let $B_{i+1}$ be the block so that $W_{i+1} = B_i \cap B_{i+1}$. By \cref{welldef}, the block $B_{i+1}$ is uniquely determined. Let $P_{i+1} = (B_1,W_1,\dots B_i, W_{i+1}, B_{i+1})$.
		\end{enumerate}		
	\end{itemize}
After the $i^{\text{th}}$ step, $\gamma(t_j)$ and $\gamma(t_{j+1})$ lie in different walls for all $j \in \{2,\dots,i-1\}$. Thus, the separating property \itemref{deftreelike}{epqilon-cond} implies that $t_{j+1} - t_{j} \ge \wallepsilon$ for all $j \in \{2,\dots,i-1\}$ and the algorithm terminates after at most $\frac{b-a}{\wallepsilon} + 2 $ steps.

 Let $P$ be the path we obtain after the algorithm terminates. First, we will show that $P = I(\gamma)$, i.e. that $P$ is the shortest geodesic path linking the unique block containing $\gamma(a)$ with a block containing $\gamma(b)$. At first, we show by induction that $P$ is a geodesic path:
The path $P_1$ is a geodesic path as it consists of a single vertex. We might assume that $P_i$ is a geodesic path and have to show that $P_{i+1}$ is a geodesic path. If we are in case 1, $I(\gamma) = P_i$ and we are done. Otherwise, $P_{i+1}=(B_1,\dots,B_{i+1})$. By induction hypothesis, the path $P_i = (B_1,\dots,B_{i})$ is a geodesic path. Hence, it suffices to show that $B_{i+1} \neq B_j$ for all $j \in \{1,\dots,i\}$. This is the case as otherwise, there exists $j \in \{1,\dots,i\}$ such that $t_{i+1}>t_j$ and $\gamma(t_{i+1})\in B_{j}$. As $B_j$ is convex, that contradicts the choice of $t_{j}$. We conclude that $P$ is a geodesic path. 

By construction, the path $P$ stars in the unique block containing $\gamma(a)$. As $P$ is a geodesic path and since the algorithm above terminates in case (1), $P$ is the shortest geodesic path linking the unique block containing $\gamma(a)$ with a block containing $\gamma(b)$. Thus $P = I(\gamma)$.

It remains to show that $I(\gamma)$ satisfies the four conditions in the claim. 
The construction of $P=I(\gamma)$ directly implies $(1)$ and $(2)$. As each block is convex, $(3)$ is satisfied. By \cref{curvebetweenwalls}, $\gamma([t_{i-1},t_{i}]) \cap \open B_{i-1} \neq \emptyset$ for all $i \in \{2,\dots,k+1\}$.	 This implies $(4)$.\end{proof}

\begin{defn}[Itineraries of geodesic segments]
	\label{lem:defitin1} If $\gamma:[a,b] \to \Sigma$ is a geodesic segment as in \cref{lem:defitin}, the geodesic path $I(\gamma)$ in $\mathcal T$ is called \textit{itinerary of $\gamma$}.
\end{defn}

	The following example shows that a geodesic segment $\gamma$ might intersect a block $B$ that does not occur in $I(\gamma)$.
		In such a case, $\gamma \cap B$ is contained a wall $W$ that does not appear in $I(\gamma)$. In the tree $\mathcal T$, $W$ is an edge that links $B$ to a block that occurs in $I(\gamma)$.

\begin{example}[See \cref{fig:treeex}]
Let $\bar B \coloneqq \{(x,0)\subseteq \R^2 \mid x \in \R\}$. For each $i \in \mathbb Z$ let $B_i \coloneqq \{(i,x) \subseteq \R^2 \mid x \in \R\}$. Let $\Sigma\coloneqq \bar B \cup \bigcup_{i \in \mathbb Z}B_i$. Then $\bar B \cup \bigcup_{i \in \mathbb Z}B_i$ is a treelike block decomposition of $\Sigma$. The set of walls is given by $\mathcal W \coloneqq \{ W_i \mid i \in \Z\}$ where $W_i ={\{(i,0)\} }$. Let $x_1 \coloneqq (0,1)$, $x_2 \coloneqq (4,0)$ and $\gamma$ be the unique geodesic segment connecting $x_1$ and $x_2$. Then $I(\gamma)=(B_0, W_0, \bar B)$ even though $\gamma$ passes through the walls $W_1, W_2 $ and $W_3$ and ends in the wall $W_4$. 
\begin{figure}[h]
	\centering
	\includegraphics[]{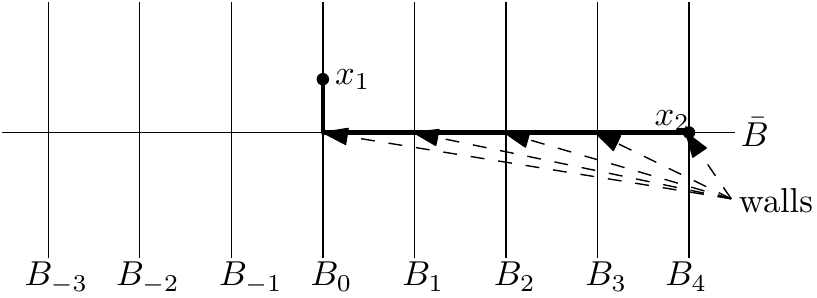}
	\caption{A CAT(0) space with a bock decomposition. Each line is a block. The intersection points of every two lines is a wall. }
	\label{fig:treeex}
\end{figure}
\end{example}

\begin{remark}
	\label{rem}
	Let $\gamma:[a,b]\to \Sigma$ be a geodesic segment as in \cref{lem:defitin}.
	\begin{enumerate}
		\item 	Each itinerary $I(\gamma)$ has an underlying graph $\bar I(\gamma)$ that consists of the vertices and (unoriented) edges in $I(\gamma)$. 
		\item 		If $\gamma$ and $\gamma'$ are two geodesic segments starting at the same point and $\gamma' \subseteq \gamma$, then $I(\gamma')$ is a subpath of $I(\gamma)$. 
	\end{enumerate}

\end{remark}

A \textit{geodesic ray} $\gamma: [0, \infty) \to \Sigma$ is an isometric embedding of $[0,\infty)$ into $\Sigma$.\newpage
\begin{defn}[Itineraries of geodesic rays]
	\label{def:itinrays}
	Let $\gamma:[0,\infty) \to \Sigma$ be a geodesic ray that doe not start in a wall. 
	\begin{enumerate}
		\item If there exists $B_0 \in \mathcal B$, $t_0 \in \R$ such that $\gamma(t) \in B_0$ for all $t \ge t_0$, then $I(\gamma)\coloneqq I(\gamma|_{[0,t_0]}).$ 
		
		\item Otherwise, there exist a sequence of times $t_0=0<t_1<t_2<\dots $ so that for every $i \in \mathbb N$, $I(\gamma|_{[0, t_{i}]})=(B_1,W_1,B_2,\dots B_i)$ with $\gamma(t_i) \in \open{B_i}$ by \cref{lem:defitin} and we define $I(\gamma)$ to be the infinite path $I(\gamma)\coloneqq(B_1,W_1,B_2,\dots).$
	\end{enumerate}
\end{defn}

We say that two geodesic rays $\gamma_1:[0,\infty)\to \Sigma$, $\gamma_2:[0,\infty)\to \Sigma$ are \textit{asymptotic} if there exists $D >0$ such that $d(\gamma_1(t), \gamma_2(t))<D$ for all $t \ge 0$. 
\begin{lemma}
	\label{lem:asymptoticrays}
Let $\gamma$ be a geodesic ray in $\Sigma$ that does not start in a wall. If $\gamma$ is asymptotic to a geodesic ray in a wall, then $I(\gamma)$ is finite. 
\end{lemma}
\begin{proof}
		Let $\gamma$ be a geodesic ray that is contained in a wall $W$. Let $\gamma'$ be a geodesic ray that does not start in a wall so that $I(\gamma')$ is infinite. We have to show that $\gamma$ and $\gamma'$ are not asymptotic, i.e. we have to prove that for each $D \in \R_{\ge 0}$ there exists $t'$ such that $d(\gamma(t),\gamma(t'))>D$. Let $D \in \R_{\ge 0}$. As $\gamma \subseteq W$, there exists $t_0 \in \R$ and a block $B \in \mathcal B$ such that $\gamma(t) \in B$ for all $t \ge t_0$. Since $I(\gamma')$ is an infinite path, there exists a block $B'$ in $I(\gamma')$ such that the unique geodesic path in $\mathcal T$ linking $B$ and $B'$ has more than $\frac{D}{\wallepsilon}$ edges. 	By \cref{lem:defitin}, there exists a time $t'$ such that $\gamma(t') \in B'$. Let $\tilde \gamma$ be the geodesic segment connecting $\gamma(t')$ with a point $b \in B$. By \cref{passthrough} and the separating property \itemref{def:blockdec}{epqilon-cond}, the length of $\tilde \gamma$ is at least $D$. Since $b$ was chosen arbitrarily and $\gamma$ ends in $B$, this implies that $d(\gamma(t),\gamma(t'))>D$.\end{proof}
 
\section{The visual boundary of every wall behaves like a cutset}
In \cref{subsec1}, we will recall the definition of visual boundaries of CAT(0) spaces. In \cref{subsec2}, we will study a cutset property of walls in CAT(0) spaces with a treelike block decomposition. 
\label{sec:visprop}

\subsection{The visual boundary of a CAT(0) space}
\label{subsec1}
Let $\Sigma$ be a complete CAT(0) space and $\base$ a chosen basepoint in $\Sigma$. Let 
\begin{align*}
	\partial \Sigma_{\base}&\coloneqq \{\alpha:[0,\infty) \to \Sigma \mid \alpha \text{ is a geodesic ray with }\alpha(0)= \base\}.
\end{align*}

If $\alpha$ is a geodesic ray starting at $\base$ and $ \epsilon>0$ $r\ge 0$, then the following sets define an open neighborhood basis for the \textit{cone topology} $\tau_{\text{cone}}$ on $\rand \Sigma_{\base}$. 

\begin{align}
	\label{defUbase2}
	U_{\base}(\alpha,r,\epsilon) &\coloneqq\{\gamma \in \rand{\Sigma_{\base}}\mid \gamma(0)=\base \text{, }d(\alpha(t), \gamma(t))< \epsilon ~\forall~ t \le r \}.
\end{align}

Recall that two geodesic rays $\gamma_1:[0,\infty)\to \Sigma$, $\gamma_2:[0,\infty)\to \Sigma$ are \textit{asymptotic} if there exists $D >0$ such that $d(\gamma_1(t), \gamma_2(t))<D$ for all $t \ge 0$. Being asymptotic is an equivalence relation and we denote the equivalence class of a geodesic ray $\gamma$ by $\as \gamma$. 
We call such an equivalence class a \textit{boundary point} and denote the set of all boundary points by
\begin{align*}
	\partial \Sigma&\coloneqq \{\alpha(\infty) \mid \alpha:[0,\infty) \to \Sigma \text{ is a geodesic ray }\}.
\end{align*}

Every equivalence class in $\rand{\Sigma}$ is represented by a unique geodesic ray starting at $\base$ (\cite[Prop. 8.2, II]{BH}). Hence, the map 
\begin{align*}
f:\partial &\Sigma_{\base}\to \partial \Sigma \\
&\alpha \to \alpha(\infty) 
\end{align*}
 is a bijection. 
 The \textit{visual boundary} of $\Sigma$ is the topological space that we obtain by pushing the topology of $\partial \Sigma_{\base}$ to $\partial \Sigma$, i.e. $A\subseteq \partial \Sigma$ is open if and only if $f^{-1}(A)$ is open in $\partial \Sigma_{\base}$. For more details see~\cite[Def.8.6 in part II]{BH}. 
 If $\Sigma$ is Gromov-hyperbolic, then the visual boundary of $ \Sigma$ coincides with the Gromov boundary of $\Sigma$. While the Gromov boundary is a quasi-isometry invariant, Croke--Kleiner~\cite{CrokeKleiner} proved that the visual boundary is not a quasi-isometry invariant. 
 
 If $B$ is a complete, convex subspace of $\Sigma$, the canonical embedding $\iota: B \embed \Sigma$ induces a topological embedding $\iota_*:~\vis Z~\to~\vis \Sigma$. For simplicity, we write $\vis B\subset \vis \Sigma$. 
\begin{lemma}[Example 8.11 (4) in Chapter II of \cite{BH}]
	Let $\Sigma$ be a complete CAT(0) space and $Z$ a complete, convex subspace of $\Sigma$. Then 	$\vis Z $ is closed in $\vis{\Sigma}$.
	\label{cor:Zlem}
\end{lemma}

\subsection{The cutset property}
\label{subsec2}
 The goal of this section is to prove the following proposition illustrated in \cref{fig:cutset property}. I would like to thank Emily Stark for the inspiration to study this property. 

\begin{prop}[cutset property]
	\label{cor:diffitinerary} 
	Let $\Sigma$ be a complete CAT(0) space with treelike block decomposition $\mathcal B$ with adjacency graph $\mathcal T$ and $\base$ a chosen basepoint in $\Sigma$ that does not lie in a wall.
		Let $\kappa$ be a connected component of $\vis \Sigma$ that contains two distinct boundary points $\gamma_1(\infty)$ and $\gamma_2(\infty)$. Let $\gamma_1$ and $\gamma_2$ be the corresponding representatives starting at $\base$.
		
		 For every wall $W$ that occurs in $I(\gamma_1)$ or $I(\gamma_2)$ but not in both $I(\gamma_1)$ and $I(\gamma_2)$, there is a geodesic ray $\gamma \subseteq W$ such that $\as \gamma \in \kappa$.\end{prop}
  \begin{figure}[h]
 	\centering
 	\includegraphics{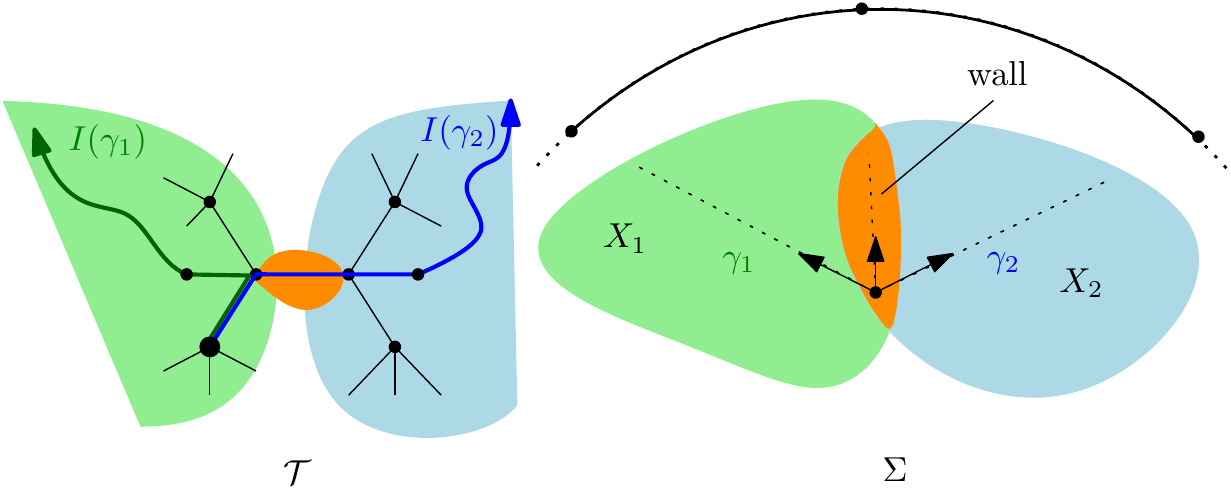}	
 	\caption{Illustration of the cutset property (\cref{cor:diffitinerary}): The orange marked edge $e$ in $\mathcal T$ to the left corresponds to a wall $W$ in $\Sigma$ which is marked orange to the right. The wall $W$ appears in $I(\gamma_2)$ but not in $I(\gamma_1)$. If we delete the edge $e$ from $\mathcal T$, $ I(\gamma_1)$ and $ I(\gamma_2)$ end in different components of the resulting graph. If we delete $W$ form $\Sigma$, the rays $\gamma_1$ and $\gamma_2$ end in different components of $\Sigma \setminus W$ and $\gamma_1(\infty)$ and $\gamma_2(\infty)$ lie in different components of $\partial \Sigma \setminus \partial W$.} 
 	\label{fig:cutset property}
 \end{figure}

For the proof, we use the following lemma for general complete CAT(0) spaces, which is similar to Lemma 3.1 in \cite{BENZVI} which deals with path-components of visual boundaries.
\begin{lemma}[Ben-Zvi--Kropholler]
	\label{lem:keyproperty}
	Let $\Sigma$ be a complete CAT(0) space and $\Sigma^1$, $\Sigma^2$ closed, convex subsets such that the intersection $W = \Sigma^1 \cap \Sigma^2$ is convex and $\Sigma = \Sigma^1 \cup \Sigma^2$. If there exist two geodesic rays $\gamma_1\subseteq \Sigma^1$ and $\gamma_2\subseteq \Sigma^2$ such that $\as{\gamma_1}$, $\as{\gamma_2} $ are contained in a connected component $ \kappa$ of $\vis \Sigma$, then there is a geodesic ray $\gamma$ in $W$ such that $\as \gamma \in \kappa$. 
\end{lemma}
\begin{proof}

Assume for a contradiction that there exists a connected component $\kappa$ in $\vis \Sigma$ containing $\gamma_1(\infty)$ and $\gamma_2(\infty)$ but no element of $\vis W$. Then $\kappa$ is a connected component of the topological subspace $\mathcal Y \coloneqq \vis \Sigma \setminus \vis W$. We use the following observations.

	\begin{enumerate}
		\item By \cref{cor:Zlem}, $\vis \Sigma^1$, $\vis \Sigma^2$ and $\vis W$ are closed in $\vis \Sigma$;\label{closed} 
			\item $\vis \Sigma^1 \setminus \vis W=\vis \Sigma \setminus \vis \Sigma^2$ and $\vis \Sigma^2 \setminus \vis W=\vis \Sigma \setminus \vis \Sigma^1$: By symmetrical reasons it suffices to prove that $\vis \Sigma^1 \setminus \vis W=\vis \Sigma \setminus \vis \Sigma^2$. \label{cor:Zlem2}		
			Let $\base$ be a basepoint in $W$. By definition,
		 \begin{align*}
		  \vis \Sigma_{\base} ^1 \setminus \vis W_{\base}&= \{\gamma \mid \gamma \text{ is a geodesic ray in }\Sigma^1 \text{ so that } \gamma(0) = \base \text{ and } \gamma \nsubseteq W\} \text{ and}\\
		  \vis \Sigma_{\base} \setminus \vis \Sigma_{\base} ^2 &= \{\gamma \mid \gamma \text{ is a geodesic ray in }\Sigma \text{ so that } \gamma(0) = \base \text{ and } \gamma \nsubseteq \Sigma^2\}.
		 \end{align*}
	
			 The space $ \vis \Sigma_{\base} ^1 \setminus \vis W_{\base}$ is a topological subspace of $\vis \Sigma_{\base} ^1$ and $ \vis \Sigma_{\base} \setminus \vis \Sigma_{\base} ^2$ is a topological subspace of $ \vis \Sigma_{\base}$. We have to show that $ \vis \Sigma_{\base} ^1 \setminus \vis W_{\base}$ and $ \vis \Sigma_{\base} \setminus \vis \Sigma_{\base} ^2$ are homeomorphic. As $\Sigma^1$ is a closed, convex subspace of $\Sigma$, 
			 the inclusion $\iota: \Sigma^1_{\base} \hookrightarrow \Sigma_{\base}$ induces a topological embedding $\iota_*: \vis \Sigma^1_{\base}\hookrightarrow \vis\Sigma_{\base}$ by \cref{cor:Zlem}. 
			It remains to show that $\iota_*(\vis \Sigma_{\base} ^1 \setminus \vis W_{\base})= \vis \Sigma_{\base} \setminus \vis \Sigma_{\base} ^2$. Let $\gamma \in \vis \Sigma_{\base} ^1 \setminus \vis W_{\base}$. The geodesic ray $\gamma$ starts in $W$ and is not contained in $W$. Since $W$ is convex, there exists a time $t_0 \in \R$ such that $\gamma(t) \in \Sigma^1\setminus W$ for all $t \ge t_0$. By assumption, $\Sigma^1 \setminus W=\Sigma \setminus \Sigma^2$. Thus, $\iota(\gamma)$ is a geodesic ray in $\Sigma$ that starts in $W$ so that $\gamma(t) \in \Sigma \setminus \Sigma^2$ for all $t \ge t_0$. Hence, $\iota(\gamma) \in \vis \Sigma_{\base} \setminus \vis \Sigma_{\base} ^2$. We conclude that $\iota_*(\vis \Sigma_{\base} ^1 \setminus \vis W_{\base})\subseteq \vis \Sigma_{\base} \setminus \vis \Sigma_{\base} ^2$. An analog argumentation implies that $\vis \Sigma_{\base} \setminus \vis \Sigma_{\base} ^2 \subseteq \iota_*(\vis \Sigma_{\base} ^1 \setminus \vis W_{\base})$: Let $\gamma \in \vis \Sigma_{\base} \setminus \vis \Sigma^2_{\base}$. The geodesic ray $\gamma$ starts in $W$ and is not contained in $W$. Since $W$ is convex, there exists a time $t_0 \in \R$ such that $\gamma(t) \in \Sigma\setminus \Sigma^2$ for all $t \ge t_0$. By assumption, $\Sigma \setminus \Sigma^2=\Sigma^1 \setminus W$. Thus, $\iota^{-1}(\gamma)$ is a geodesic ray in $\Sigma^1$ that starts in $W$ so that $\gamma(t) \in \Sigma^1 \setminus W$ for all $t \ge t_0$. Hence, $\iota^{-1}(\gamma) \in \vis \Sigma_{\base}^1 \setminus \vis W$.
		\item $\vis \Sigma^1 \setminus \vis W$ and $\vis \Sigma^2 \setminus \vis W$ are open in $\vis \Sigma$:\label{open}
		Indeed, let $i,j \in \{1,2\}$, $i \neq j$.
		As $\vis \Sigma^1$, $\vis \Sigma^2$ are closed in $\vis \Sigma$, $\vis \Sigma \setminus \vis \Sigma^1$ and $\vis \Sigma \setminus \vis \Sigma^2$ are open in $\vis \Sigma$. By (\ref{cor:Zlem2}) $\vis \Sigma^i \setminus \vis W=\vis \Sigma \setminus \vis \Sigma^j$ for $i, j \in \{1,2\}$, $i \neq j$. Hence, $\vis \Sigma^1 \setminus \vis W$ and $\vis \Sigma^2 \setminus \vis W$ are open in $\vis \Sigma$. 
	\end{enumerate}
		Let $i \in \{1,2\}$. By (\ref{closed}), $\vis \Sigma^i \cap \mathcal Y$ is closed in $\mathcal Y$. 
	Since $\vis \Sigma^1 \cap \mathcal Y= \vis \Sigma^1 \cap (\vis \Sigma \setminus \vis W)= \vis \Sigma^1 \setminus \vis W $, 	
	$\vis \Sigma^1 \cap \mathcal Y$ is open in $\mathcal Y$ by (\ref{open}). Analogously, $\vis \Sigma^2 \cap \mathcal Y$ is open in $\mathcal Y$. Thus, $\vis \Sigma^1 \cap \mathcal Y$ and $\vis \Sigma^2 \cap \mathcal Y$ are open and closed. This implies that $\gamma_1(\infty)$ and $\gamma_2(\infty)$ lie in different connected components since $\gamma_1(\infty)\in \vis \Sigma^1 \cap \mathcal Y$ and $\gamma_2(\infty)\in \vis \Sigma^2 \cap \mathcal Y$ -- a contradiction to the connectedness of $\mathcal \kappa$.\end{proof}

\begin{proof}[Proof of \cref{cor:diffitinerary} ]
	Let $\gamma_1$, $\gamma_2$ and $\kappa$ as in the claim. 
	Suppose that $W$ is a wall that appears in $I(\gamma_1)$ or $I(\gamma_2)$ but not in both. Let $T_1$ and $T_2$ be the two subtrees of the adjacency graph $\Tblock$ of $\mathcal B$ we obtain by removing $W$ from $\Tblock$. Let $\Sigma^i \coloneqq \bigcup_{B \in V(T_i)}B$. Then $\Sigma = \Sigma^1 \cup \Sigma^2$ and $\Sigma^1\cap \Sigma^2 = W$. By \cref{convex}, the spaces $\Sigma^1$ and $\Sigma^2$ are closed and convex. The wall $W$ is closed and convex by \cref{wallconvex}. We will show that one of the two geodesic rays $\gamma_1$ and $\gamma_2$ ends in $\Sigma^1 $ and that the other one ends in $\Sigma^2$. Then the claim follows by applying \cref{lem:keyproperty} with $\Sigma = \Sigma^1 \cup \Sigma^2$. 
	
 We assume without loss of generality that $W$ appears in $I(\gamma_2)$ but not in $I(\gamma_1)$. Let $B_1 \in V(T_1)$ and $B_2\in V(T_2)$ be the blocks so that $W = B_1 \cap B_2$ and so that $B_2$ occurs after $B_1$ in $I(\gamma_2)$. By \cref{lem:defitin}, there are two times $t_1$, $t_2$, $t_1 < t_2$ so that $\gamma_2(t_1) \in \open{B_1}$ and $\gamma_2(t_2) \in \open{B_2}$. By definition of $\Sigma^1$ and $\Sigma^2$, $\gamma_2(t_1) \in \Sigma^1 \setminus \Sigma^2$ and $\gamma_2(t_2) \in \Sigma^2 \setminus \Sigma^1$. Assume for a contradiction that there exists $t_3> t_2$ such that $\gamma_2(t_3) \in \Sigma^1\setminus \Sigma^2$. Then $\gamma_2(t_1) \in \Sigma^1$, $\gamma_2(t_3) \in \Sigma^1$ and $\gamma_2(t_2) \notin \Sigma^1$ and $\gamma_2([t_1,t_3])$ is a geodesic segment connecting two points in $\Sigma_1$ that contains the point $\gamma_2(t_2)$ outside of $\Sigma_1$ -- a contradiction to the convexity of $\Sigma^1$. Thus, $\gamma_2(t) \in \Sigma^2$ for all $t \ge t_2$, i.e. $\gamma_2(t)$ ends in $\Sigma^2$. 
 
 	It remains to show that $\gamma_1$ ends in $\Sigma^1$. Since geodesic paths in trees are unique, each path linking a block in $T_1$ with a block in $T_2$ passes through the wall $W$. Since $W$ does not appear in $I(\gamma_1)$,  the geodesic path $ I(\gamma_1)$ is either contained in $T_1$ or in $T_2$. As $\gamma_1$ and $\gamma_2$ start at the same point, $I(\gamma_1)$ and $I(\gamma_2)$ start with the same block $B_0$. 
 		Recall that $W= B_1\cap B_2$ and that $B_2$ appears after $B_1$ in $I(\gamma_2)$. As $I(\gamma_2)$ is a geodesic path, $W$ does not occur twice in $I(\gamma_2)$ and the subpath $(B_0,\dots,B_1) \subseteq I(\gamma_2)$ does not contain $W$. Thus, as $B_1 \in T_1$ and $B_0 \in V(T_1)$, the itinerary $I(\gamma_1)$ is a path in $T_1$. By \cref{lem:defitin}, $\gamma_1([0,\infty)) \subseteq \Sigma_1$. In particular, $\gamma_1$ ends in $\Sigma_1$.\end{proof}

\section{Key properties of the Morse boundary}
\label{key_Morse}
In \cref{sec1Morse}, we will recall the definition of Morse boundaries and will transfer the observations of \cref{sec:visprop} to Morse boundaries. In \cref{sec:Morse2}, we will prove that geodesic rays of infinite itinerary are lonely. This is the only point, where we use the Morse-property in this paper. In \cref{subsecMorse3}, we will study relative Morse boundaries.
\subsection{From the visual boundary to the Morse boundary}
\label{sec1Morse}
In this section, we shortly recap the definition of the contracting boundary and the Morse boundary and obtain as a consequence that the cutset property in  \cref{cor:diffitinerary} holds not only for the visual boundary but also for the Morse boundary. 

Let  $\Sigma$ be a complete CAT(0) space and $C$ be  a convex subset that is complete in the induced metric. Then there is a well-defined nearest point projection map $\pi_C:\Sigma \to C$. This projection map is continuous and does not increase distances (See~\cite[Prop. 2.4 in II ]{BH}).
\begin{defn}[contracting geodesics]
	Given a fixed constant $D$, a geodesic ray or geodesic segment $\gamma$ in a complete CAT(0) space $(\Sigma,d)$ is said to be \textit{$D$-contracting} if for all $x$, $y \in \Sigma$, 
	\begin{align*}
		d(x,y) < d(x,\pi_{\gamma}(x)) \Rightarrow d(\pi_{\gamma}(x), \pi_\gamma(y))<D.
	\end{align*}	
	We say that $\gamma$ is \textit{contracting} if it is $D$-contracting for some $D$. 
\end{defn}

Charney--Sultan \cite{CharSul} introduced a quasi-iosmetry invariant of complete CAT(0) spaces, called \textit{contracting boundary}. Let $\Sigma$ be complete CAT(0) space. The underlying set of the \textit{contracting boundary of $\Sigma$} is the set 
\begin{align*}
	\partial_c \Sigma&\coloneqq \{\as \alpha \mid \alpha \text{ is a contracting geodesic ray}\}.
\end{align*} By definition, $\partial_c \Sigma \subseteq \vis \Sigma$ (as sets). Let $(\partial_c \Sigma, \tau_{\text{cone}})$ be the set $\partial_c \Sigma$ equipped with the subspace topology of $\vis \Sigma$. Cashen~\cite{Cashen} proved that $(\partial_c \Sigma, \tau_{\text{cone}})$ isn't a quasi-isometry invariant. For obtaining a quasi-isometry invariant-topology, we choose a basepoint $\base$ in $\Sigma$ and define 
\begin{align*}
	\partial_c \Sigma_{\base}&\coloneqq \{ \alpha:[0,\infty) \to \Sigma \mid \alpha \text{ is a contracting geodesic ray and }\alpha(0)= \base\}.
\end{align*} 
As before in \cref{subsec1}, \[f: \partial_c \Sigma_{\base}\to \partial_c \Sigma,~ \alpha \mapsto \alpha(\infty)\] is a bijection. Let
\begin{align*}
	\partial_c^N\Sigma_{\base}&\coloneqq \{\alpha:[0,\infty) \to \Sigma\mid \alpha\text{ is a }N\text{-contracting geodesic ray and }\alpha(0)=\base\}.
\end{align*}
Now, let
\begin{align*}
	(\partial_c \Sigma_{\base}, \tau_{\text{dirlim}})\coloneqq \lim_{\overrightarrow N}(	\partial_c^N\Sigma_{\base}, \tau_{\text{cone}}),
\end{align*}
i.e. a set a $O$ is open in $\tau_{\text{dirlim}}$ if for each $N \in \N$, $O$ is open in $ O \cap \partial_c^N\Sigma_{\base}$. 

The \textit{contracting boundary} $\partial_c \Sigma$ of $\Sigma$ is the topological space that we obtain by pushing the topology $\tau_{\text{dirlim}}$ of $\partial_c \Sigma_{\base}$ to $\partial_c \Sigma$, i.e. $A\subseteq \partial_c \Sigma$ is open if and only if $f^{-1}(A)$ is open in $\partial_c \Sigma_{\base}$. If $\Sigma$ is Gromov-hyperbolic, then $\partial_c\Sigma$ coincides with the Gromov boundary of $\Sigma$.\\

Cordes \cite{Cordes_properMorse} generalized the contracting boundary to a quasi-isometry invariant of proper geodesic metric spaces, called the \textit{Morse boundary}. This generalization is based on the following characterizations of contracting geodesic rays in complete CAT(0) spaces. 

\begin{defn}
	A function $M:[1, \infty) \times [0, \infty) \to [0,\infty)$ is called a \textit{Morse gauge}. 
\end{defn}

\begin{defn}
	Let $\gamma:[0,\infty)\to \Sigma$ be a geodesic ray in a proper geodesic metric space $\Sigma$. 	Given a Morse-gauge $M$, $\gamma$ is \textit{M-Morse} if, for every $K\ge 1$, $L\ge 0$, every $(K,L)$-quasi-geodesic $\sigma$ with endpoints on $\gamma$ is contained in the $M$-neighborhood of $\gamma$. 
\end{defn}

\begin{defn}
	Let $\Sigma$ be a complete CAT(0) space.
	A geodesic ray $\gamma$ is \textit{slim} if there exists $\delta>0$ such that for all $x \in \Sigma$, $y \in \gamma$, the distance of $\pi_{\gamma}(x)$ and the geodesic segment connecting $x$ and $y$ is less than $\delta$. 
\end{defn}

Sultan~\cite{Sultan} and Charney-Sultan~\cite{CharSul} showed: 
\begin{lemma}[Charney--Sultan]\label{contractingandmorse}\label{slimness}
	Let $\gamma$ be a geodesic ray in a complete CAT(0) space. The following are equivalent: 
	\begin{enumerate}
		\item $\gamma$ is slim 
		\item $\gamma$ is Morse 
		\item $\gamma$ is contracting.
	\end{enumerate} 
\end{lemma}

 The \textit{Morse boundary} $\morse \Sigma$ of a proper geodesic metric space $\Sigma$ is a topological space with underlying set 
\begin{align*}
	\morse \Sigma&\coloneqq \{\as \alpha \mid \alpha \text{ is a Morse geodesic ray}\}.
\end{align*}
The topology of $\morse \Sigma$ is a direct-limit topology so that in the CAT(0)-case, $\morse \Sigma$ and $\partial_c \Sigma$ are homeomorphic: 
 Suppose that $\base$ is a basepoint in the proper geodesic metric space $\Sigma$ and $N$ is a Morse gauge. Let 
\begin{align*}
	\partial^{N}_M\Sigma_{\base}&\coloneqq 	\{\as \alpha \mid \exists \beta \in \alpha(\infty)\text{ that is an $N$-Morse geodesic ray with } \beta(0)=\base\}
\end{align*}
endowed with the compact-open topology.

The following is Lemma 3.1 in \cite{Cordes_properMorse}:
\begin{lemma}\label{sets}
	Let $\Sigma$ be a proper geodesic metric space and $\base \in  \Sigma$. For $N = N(K,L)$ a Morse gauge, let $\delta_N \coloneqq \max\{4N(1,2N(5,0))+2N(5,0), 8N(3,0)\}$. Let 
	$\alpha:[0,\infty) \to \Sigma$ be a $N$-Morse geodesic ray with $\alpha(0)= \base$ and for each positive integer $n$ let $V_n(\alpha)$ be the set of geodesic rays $\gamma$ such that $\gamma(0)= \base$ and $d(\alpha(t), \gamma(t)) < \delta_N$ for all $t < n$. Then 
	\begin{align}
		\{V_n(\alpha) \mid n \in \N\} 
	\end{align}
	is a fundamental system of (not necessarily open) neighborhoods of $\alpha(\infty)$ in $\partial_M^N \Sigma_{\base}$. 
\end{lemma}
By Proposition 3.12 in \cite{Cordes_properMorse}, $\partial^{N}_M\Sigma_{\base}$ is compact for each Morse gauge $N \in \mathcal M$.

Let $\mathcal M$ be the set of all Morse gauges. If $N$ and $N'$ are two Morse gauges, we say that $N \le N'$ if and only if $N(\lambda, \epsilon) \le N'(\lambda, \epsilon)$. This defines a partial ordering on $\mathcal M$. 
Corollary 3.2 in \cite{Cordes_properMorse} and the proof of Proposition 4.2 in \cite{Cordes_properMorse} implies 
\begin{lemma}
	Let $N$ and $N'$ two Morse gauges such that $N \le N'$ and $\Sigma$ a proper geodesic metric space with basepoint $\base$. Then the associated inclusion map $\iota_{N,N'}:  \morse^N \Sigma_{\base} \embed  \morse^{N'} \Sigma_{\base}$ is a topological embedding. \label{goodembeddings}
\end{lemma}
Now, let 

	\[\morse \Sigma\coloneqq \underset{\underset{N\in \mathcal M}{\longrightarrow}}{\lim}  \morse^{N} \Sigma_{\base}.\]
	with the induced direct limit topology, ie. a set $U$ is open in $\morse \Sigma$ if and only if $U \cap \partial_M\Sigma_{\base}^{N}$ is open in $\partial_M\Sigma_{\base}^{N}$ for all $N \in \mathcal M$. Cordes proves that this is well-defined, i.e. he shows independence of the basepoint. 




Finally, we will use the following consequence of the Theorem of Arzelà-Ascoli. It is Corollary 1.4 in \cite{Cordes_properMorse} and conform with \cite{Topology}.
\begin{lemma}[Arzelà-Ascoli]
	\label{Ascoli}
	Let $\Sigma$ be a proper metric space and $p \in \Sigma$. Then any sequence of geodesics $\gamma_n:[0,L_n] \to \Sigma$ with $\gamma_n(0)= p$ and $L_n \to \infty$ has a subsequence that converges uniformly on compact sets to a geodesic ray $\gamma:[0,\infty) \to \Sigma$.  
\end{lemma}


The topology of the Morse boundary is finer than the subspace topology of the visual boundary, i.e. if a set is open in the subspace topology of the visual boundary, then it is also open in the Morse boundary. Thus, 	\cref{cor:diffitinerary} implies

\begin{cor}[cutset property]
	\label{cor:diffitinerary2}
	Let $\Sigma$ be a complete CAT(0) space with treelike block decomposition $\mathcal B$ with adjacency graph $\mathcal T$ and $\base$ a chosen basepoint in $\Sigma$ that does not lie in a wall. 
	Let $\kappa$ be a connected component of $\morse \Sigma$ that contains two distinct boundary points $\gamma_1(\infty)$ and $\gamma_2(\infty)$. 
	Let $\gamma_1$ and $\gamma_2$ be the corresponding representatives starting at $\base$. 
	
	For every wall $W$ that occurs in $I(\gamma_1)$ or $I(\gamma_2)$ but not in both $I(\gamma_1)$ and $I(\gamma_2)$, there is a geodesic ray $\gamma \subseteq W$ such that $\as \gamma \in \kappa$.
\end{cor}

\subsection{Loneliness of Morse geodesic rays with infinite itinerary}
\label{sec:Morse2}
Let $\Sigma$ be a proper CAT(0) space with treelike block decomposition $\mathcal B$ with adjacency graph $\mathcal T$ and $\base$ a chosen basepoint in $\Sigma$ that does not lie in a wall.
The goal of this subsection is to prove that geodesic rays starting at $\base$ of infinite itinerary are lonely.  I would like to thank Tobias Hartnick for his help to simplify the proof for this property. 	

\begin{prop}[Loneliness property]
	\label{Morse_conn_of_type_Ainf}
	Let $\alpha$ and $\beta$ be two distinct geodesic rays in $\Sigma$ starting at $\base$ that have infinite itinerary. If at least one of both is Morse, then $I(\alpha) \ne I(\beta)$. 	
\end{prop}

This property is quite remarkable as it is the only point in this paper where we use the Morse-property. The Loneliness property does not hold for non-Morse geodesic rays in general.  The following example shows that visual boundaries might contain infinitely many geodesic rays that are pairwise non-asymptotic and have all the same infinite itinerary.

\begin{example}
	\begin{figure}[h]
	\centering
	\includegraphics[]{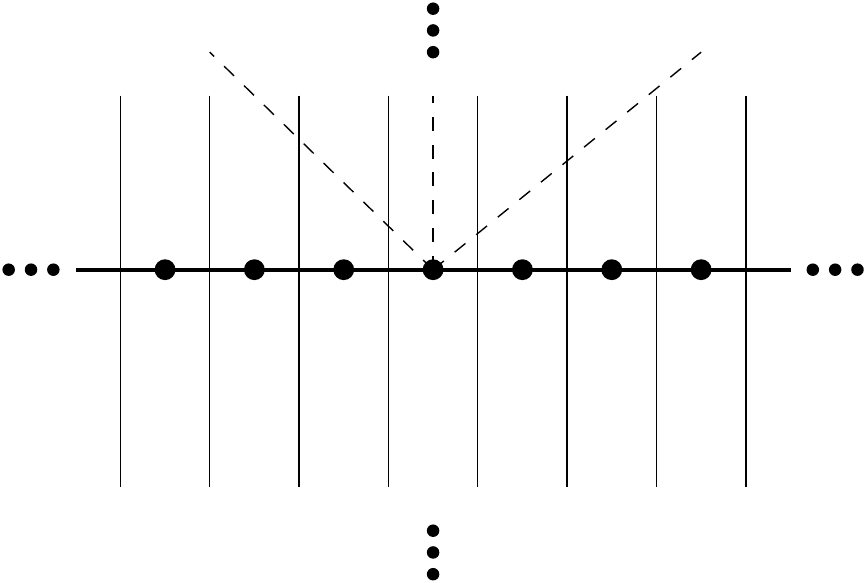}
	\caption{A block decomposition of the Euclidean plane in which the itinerary of each geodesic ray starting in the interior of a block is either trivial (i.e. a path consisting of one vertex) or an infinite path. The dashed lines denote three geodesic rays with different itineraries. }
	\label{fig:planeexample}
\end{figure}
	Let $\Sigma = \R^2$ and $\mathcal B \coloneqq \{[i,i+1]\times \R \mid i \in \Z\}$ as pictured in \cref{fig:planeexample}. For $i \in \Z$, let $B_i \coloneqq [i,i+1]\times \R$. The adjacency graph of $\mathcal B$ is a bi-infinite path of the form $(\dots, B_{-1}, B_{0}, B_{1}, \dots)$. 
	Let $\base \coloneqq (\frac{1}{2},0)$. A geodesic ray $\gamma$ starting at $\base$ can be of three different kinds:
	If $\gamma$ is parallel to the $Y$-axes, $\gamma$ is contained in the block $B_0$, i.e. its itinerary is the path that consists of the block $B_0$. 
	If $\gamma$ intersects the $Y$-axis $\R \times 0$, the itinerary of $\gamma$ is the infinite path $(B_0,B_{-1},\dots)$. 
	In the remaining case, $\gamma$ is not parallel to the $Y$-axis and does not intersect the $Y$-axis. In this situation, the itinerary of $\gamma$ is the infinite path $(B_0,B_{1},\dots)$.
\end{example}

The following proof is inspired by the example of Charney--Sultan discussed in \cref{subsecRACGIntro} and uses methods of the proof of Proposition 3.7 in~\cite{CharSul}.

\begin{proof}[Proof of \cref{Morse_conn_of_type_Ainf}]
	 \begin{figure}[h]
		\centering
		\includegraphics{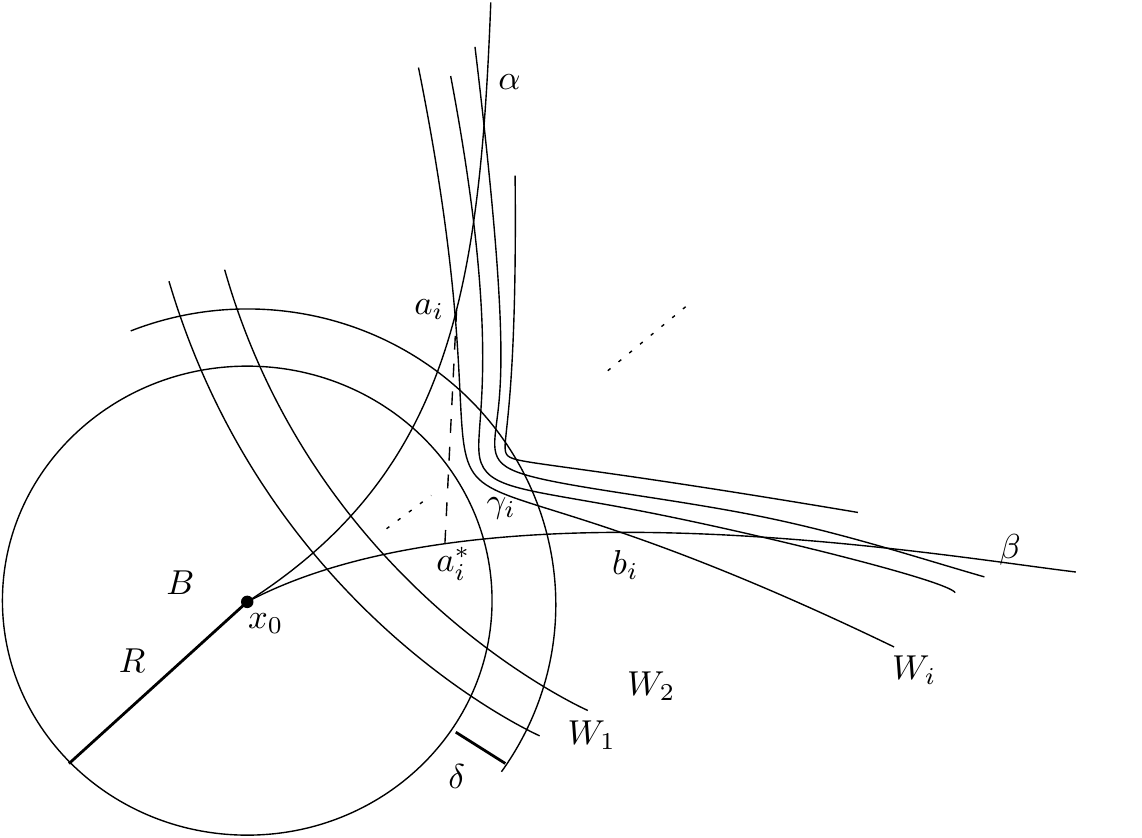}
		
		\caption{Illustration of the proof of the Loneliness property. }
		\label{fig:lonelinessproperty}
	\end{figure}
	The proof is illustrated in \cref{fig:lonelinessproperty}.
	Assume for a contradiction that $\alpha$ and $\beta$ are two geodesic rays starting in $\Sigma$ such that 
	\begin{itemize}
		\item $\alpha(0)=\beta(0)=\base$; and
		\item $\alpha \neq \beta$; and
		\item $I\coloneqq I(\alpha)=I(\beta)$; and
		\item $I$ is an infinite path; and
		\item $\beta$ is Morse.
	\end{itemize}
Let
	\begin{itemize}
	\item $(W_i)_{i \in \N}$ be the sequence of consecutive walls that are contained in $I$;
	\item $(a_i)_{i \in \N}$ a sequence of points such that $a_i \in W_i \cap \alpha$.
	\item $(b_i)_{i \in \N}$ a sequence of points such that $b_i \in W_i \cap \beta$.
	\item $(\gamma_i)_{i \in \N}$ the sequence of geodesic segments $\gamma_i$ connecting $a_i$ with $b_i$.
	\item $(a_i^*)_{i \in \N}$ be the sequence of projection points 	
	 $a_i^*\coloneqq \pi_\beta(a_i)$.
\end{itemize}

By proposition 3.7 (2) in \cite{CharSul}, there exists $R>0$ such that $\{a_i^*\mid i \in \mathbb N\} \subseteq B_R(\base)$, where $B_R(\base)$ denotes the closed $R$-ball about $\base$. 
For $i \in \mathbb N$, let $\Delta_i\coloneqq \Delta(a_i,b_i,a_i^*)$ be the geodesic triangle in $\Sigma$ with corners $a_i$, $b_i$ and $a_i^*$. As $\beta$ is Morse, $\beta$ is slim by \cref{slimness}. Hence, there exists $\delta>0$ such that $U_{\delta}(a_i^*)\cap \gamma_i\neq \emptyset$. As $a_i^* \in B_R(\base)$ it follows that $\gamma_i \cap B_{R+\delta}(\base)\neq \emptyset$. 

Recall that for each $i \in \mathbb N$, $a_i$ and $b_i$ are contained in a wall. As each wall is convex, $\gamma_i \subseteq W_i$. Thus, $W_i \cap B_{R+\delta(\base)}\neq \emptyset$ for all $i \in \N$.
We conclude that infinitely many walls intersects the ball $B_{R + \delta}(x_0)$. But this is impossible because of \cref{ball}.\end{proof}

\subsection{Relative Morse boundaries of convex subspaces}
\label{subsecMorse3}

Let $\Sigma$ be a proper geodesic metric space and $B$ a convex, complete subspace whose Morse boundary is \textit{$\sigma$-compact},  i.e. a union of countably many compact subspaces. Recall that $\relMorse{\Sigma}{B}$ denotes the \textit{relative Morse boundary of $B$ in $\Sigma$}, i.e.  the subset of $\morse \Sigma$ that consists of all equivalence classes of geodesic rays in $B$ that are Morse in the ambient space $\Sigma$. In this section, we study the relation of the two topological spaces that are obtained by endowing $\relMorse{\Sigma}{B}$ with the subspace topology of $\morse B$ and $\morse \Sigma$. Our goal is to proof \cref{blocklemma} and \cref{corOpenset}. 

I would like to thank Nir Lazarovich for his help to improve this section. Moreover,  I would like to thank Elia Fioravanti for sending me an example showing that the inverse of the embedding in the following lemma need not be continuous (see \cref{exNotcontinuous}).



	

\begin{lemma}
	\label{blocklemma}
	Let $\Sigma$ be a proper geodesic metric space with $\sigma$-compact Morse boundary. 
	Let $B$ be a complete, convex subspace of $\Sigma$ that contains all geodesic rays in $\Sigma$ that start in $B$ and are asymptotic to a geodesic ray in $B$. 
If we endow $\relMorse{\Sigma}{B}$ with the subspace topology of $\morse \Sigma$, then the map
	\begin{align*}
		\iota_*:\relMorse{\Sigma}{B} \embed \morse B \\
		\gamma(\infty)\to \gamma(\infty)
	\end{align*}
	is continuous.
\end{lemma}

\begin{cor}
	\label{corOpenset}
	Let $B$ be a closed convex subspace of a proper CAT(0) space $\Sigma$. 
If  $\relMorse{\Sigma}{B}$ endowed with the subspace topology of $\morse B$ is totally disconnected, then $\relMorse{\Sigma}{B}$ endowed with the subspace topology of $\morse \Sigma$ is totally disconnected.
\end{cor}


\begin{proof}[Proof of \cref{corOpenset}]
	If the  assumptions of \cref{blocklemma} are satisfied, then the continuity of $\iota_*$ implies  \cref{corOpenset}. Hence it remains to verify the assumptions of \cref{blocklemma}: The Main theorem in \cite{CharSul} implies that $\morse \Sigma$ is $\sigma$-compact. Moreover, $B$ contains all geodesic rays in $\Sigma$ that start in $B$ and are asymptotic to a geodesic ray in $B$ because $\Sigma$ is CAT(0) and $B$ is closed and convex. 
\end{proof}



For proving \cref{blocklemma}, we need the following facts about direct limits. 
The following lemma is Lemma 3.10 in \cite{CharSul}. For completeness, we cite their proof as well. 
\begin{lemma}[Charney--Sultan]
	\label{dirlim1}
	Let $X=\underset{\longrightarrow}{\lim} X_i$, $Y = \underset{\longrightarrow}{\lim}  Y_i$ with the direct limit topology. Suppose $f:X \to Y$ is a function so that $f(X_i) \subseteq Y_{g(i)}$ where $g: \N \to \N$ is a non-decreasing function. If  $f_i: X_i \to Y_{g(i)}$, $x \mapsto f(x)$ is continuous for all $i$, then $f$ is continuous. 
\end{lemma}
\begin{proof}
	Let $U\subseteq Y$ be open. Then $U_{g(i)}= U \cap Y_{g(i)}$ is open in $Y_{g(i)}$ for all $i$. Since $f_i$ is continuous, if follows that $f^{-1}(U)\cap X_i = f^{-1}(U_{g(i)}) \cap X_i = f_i^{-1}(U_{g(i)})$ is open in $X_i$. By definition of the direct limit topology, $f^{-1}(U)$ is open in $X$. 
\end{proof}
In the following lemma, we study the direct limit of countably many topological spaces.
\begin{lemma}\label{dirlim2}
	Let $X=\underset{\underset{i\in \N}{\longrightarrow}}{\lim}  X_i$ be a direct limit of topological spaces $X_i$ with associated topological embeddings $\iota_{i,j}:X_i \to X_j$ where $i, j \in \mathbb N$, $i \le j$. Let $B \subseteq X$ be a closed subspace of $X$. We equip $B \cap X_i$ with the supspace topology of $X_i$ for all $i$.
	Then   $\underset{\underset{i\in \N}{\longrightarrow}}{\lim}  (X_i\cap B)$ is homeomorphic to $B$ equipped with the subspace topology of $X$. \end{lemma}
\begin{proof}
		\phantom\qedhere
	Let $O$ be an open set in $ B$ equipped with the subspace topology of $X$. 
	By \cref{dirlim1}, $O$ is open in the direct limit $\underset{\underset{i\in \N}{\longrightarrow}}{\lim}  (X_i\cap B)$.	
	
	Now let $O \subseteq B$ be an open set in $\underset{\underset{i\in \N}{\longrightarrow}}{\lim}  (X_i\cap B)$.
	We have to find a set $\tilde O$ that is open in $X$ and satisfies $\tilde O \cap B = O \cap B$. Since  each $\iota_{i,j}:X_i \to X_j$ is a topological embedding for each $i, j \in \mathbb N$, we may assume that $X_i \subseteq X_j$ for all $i \le j$. 
	For every $i \in \N$, we will define a set $\tilde O_i \subseteq X_i$ so that
		\setcounter{equation}{0}
	\begin{align}
		&\tilde O_i \text{ is open in }X_i \label{prop1},\\
		&\tilde O_i \subseteq \tilde O_j \text{ for all }i \le j \label{prop2},\\
		&\bigcup_{i \in \N}{\tilde O_i \cap B} = O \cap B \label{prop3}.
	\end{align}

	Then the set  $\tilde O \coloneqq \bigcup_{i \in \N}{O_i}$  is the set, we are looking for. Indeed, $\tilde O \cap B = O \cap B$ by the listed properties. Moreover, $\tilde O$ is open in $X$:  Let $k \in \N$. We have to show that $\tilde O \cap X_k$ is open in $X_k$ for all $k \in \N$. By the listed properties, $\tilde O_i \subseteq \tilde O_k$ for all $i \le k$. Hence, $\tilde O \cap X_k =  \bigcup_{i \ge k }{O_i\cap X_k} $. By assumption, $\iota_{k,i}: X_k \embed X_i$ is continuous for all $i \ge k$. Hence, $\inv{\iota_{k,i}(O_i)}= O_i \cap X_k$ is open in $X_k$ for all $i \ge k$. Hence  $\tilde O \cap X_k =  \bigcup_{i \ge k }{O_i\cap X_k} $ is open in $X_k$ as union of open sets in $X_k$.\\
	
	It remains to define the sets $\tilde O_i, i \in \N$. Let $i \in \mathbb N$. In step 1, we will prepare the definition of $\tilde O_i$. In step 2, we will define $\tilde O_i$. In step 3, we will show that $\tilde O_i$ satisfies (\ref{prop1}),  (\ref{prop2}) and  (\ref{prop3}).
	\begin{enumerate}
		\item[\textbf{Step 1}:] Recall that $i \in \mathbb N$ is arbitrarily chosen. 
		Since $O \subseteq B$ is an open set in $\underset{\underset{i\in \N}{\longrightarrow}}{\lim}  (X_i\cap B)$, there exists an open set $ O_i$ in $X_i$ such that $ O_i  \cap B = O \cap X_i.$ Let $O_i^i \coloneqq O_i$. For each $j \in \N$, $j \ge i$, we will define a set $O_i^j \subseteq X_j$ inductively such that 
		\begin{enumerate}
			\item $O_i^j \cap X_i = O_i$,
			\item $O_i^j$ is an open set in $X_j$,
			\item $O_i^k \subseteq O_i ^j$ for all $k \le j.$\\
			
			\item[Induction base:]Let $O_i^i \coloneqq O_i$. 
			\item[Induction step:]Suppose that $O_i^j$ is a set in $X_j$ with the properties listed above. \\Since $\iota_{j,j+1}~:~X_j~\embed~X_{j+1}$ is a topological embedding, 
			$O_i^j$ is an open set in $X_j$ equipped with the subspace topology of $X_{j+1}$. Thus, there exists a set $O_i^{j+1}$ that is open in $X_{j+1}$ so that 
			$O_i^{j+1}\cap X_j = O_i^j$. In particular, $O_i^j \subseteq O_i^{j+1}$. Moreover, $O_i^{j+1}\cap X_i = O_i^{j+1}\cap X_j \cap X_i = O_i^j \cap X_i = O_i$ where the last equality follows by induction hypothesis. 
		\end{enumerate}

		\item[\textbf{Step 2}:] For each $i \in I$, we define
		\begin{align*}		
			&\tilde O_i\coloneqq \bigcup_{j \le i}  O_j^i\setminus A \text{, where } A \coloneqq B \setminus O \text{.  }
		\end{align*}
		
		\item[\textbf{Step 3}:] Let $i \in \N$. It remains to show that $\tilde O_i$ satisfies (\ref{prop1}), (\ref{prop2}) and (\ref{prop3}).
		\begin{enumerate}
			\item[(\ref{prop1}):]  We have to show that $\tilde O_i$ is open in $X_i$. By definition, \[\tilde O_i = \bigcup_{j \le i}  O_j^i\setminus A = \bigcup_{j \le i}  O_j^i \cap (X_i \setminus A).\]  It remains to show that  $O_j^i$ and  $X_i \setminus A$ are open in $X_i$ for all $j \le i$.  By definition, $O_j^i$ is open in $X_i$ for all $j \le i$. Since $B$ is closed in $X$, $A= B \cap (X \setminus O)$ is closed in $X$. This implies that $X_i \cap A$ is closed in $X_i$  by definition of the direct limit topology. Hence, $X_i \setminus A$ is open in $X_i$. 
			\item[(\ref{prop2}):] Let $i \le k$. Then $\tilde O_i = \bigcup_{j \le i}  O_j^i\setminus A \subseteq   \bigcup_{j \le k}  O_j^k\setminus A = \tilde O_k$ since  $O_j^i \subseteq O_j^k$ for all $j \le i \le k$.  
			\item[(\ref{prop3}):] We have to show that $\bigcup_{i \in \N}{\tilde O_i \cap B} = O \cap B$.
			Since  $B= O\sqcup  A$ is the disjoint union of $O$ and $A$ and $B \subseteq O$, we have that 
			\begin{align*}
				\tilde O \cap B = \bigcup_{i \in\N }{\tilde O_i \cap B} =\bigcup_{i \in \N}{\bigcup_{j \le i}{( O_j^i \setminus A) \cap B}}=\bigcup_{i \in \N}{\bigcup_{j \le i}{ O_j^i \cap O}}.
			\end{align*}
		
			 Thus, $\tilde O \cap B\subseteq O$. 
			On the other hand, $O \subseteq \tilde O \cap B$ because $ \bigcup_{i \in \N}{\bigcup_{j \le i}{ O_j^i B\cap O}} $ contains $O_i^i=O_i $ for all $i \in \N$ and $O = \bigcup_{i \in \N}O_i$.\qed
		\end{enumerate}
	\end{enumerate}
\end{proof}

For applying \cref{dirlim2} to the relative Morse boundary of $B$ in $\Sigma$, we need the following variant of  Example 8.11 (4) in Chapter II of \cite{BH} (see \cref{cor:Zlem}) for Morse boundaries.

\begin{lemma}
	\label{closedisclosed}
	If $B$ is a complete, convex subspace of  a proper geodesic metric space $\Sigma$, then $\relMorse{\Sigma}{B}$ is closed in $\morse \Sigma$.
\end{lemma}

\begin{proof}[Proof of \cref{closedisclosed}]
	Let $p$ a basepoint in $B$ and  $\mathcal M$ be the set of all Morse gauges. We have to show that $\relMorse{\Sigma}{B} \cap \morse^{N}\Sigma_{\base}$ is closed $\morse^{N}\Sigma_{\base}$ for all $N \in \mathcal M$. 
	Let $(\beta_n(\infty))_{n \in \N}$ be a sequence of equivalence classes of geodesic rays in $\relMorse{\Sigma}{B} \cap \morse^{N}\Sigma_{\base}$ and $(\beta_n)_{n \in \N}$ be corresponding representatives that start at $p$ and lie in $B$. By  the Theorem of Arzelà-Ascoli \ref{Ascoli}, the sequence $(\beta_n)_{n \in \N}$ has a convergent subsequence $(\beta_{n_m})_{m \in \N}$. As $B$  is complete, $\beta \coloneqq \lim_{m \to \infty}{\beta_{n_m}}$ is a geodesic ray in $B$. Moreover, $\beta$ is $N$-Morse. Indeed, Let $q$ be a point on a quasi-geodesic with endpoints on $\beta$ and $t \in \R_{\ge 0}$ such that $d(\beta,q) = d(\beta(t), q)$.  The continuity of the distance function $d$ and the $N$-Morseness of $\beta_{n_m}$, $m \in \N$ implies that $d(q, \beta(t)) = d(q, \lim_{m \to \infty}\beta_{n_m}(t))= \lim_{m \to \infty} d(q, \beta_{n_m})\le N$. 
\end{proof}

\begin{proof}[Proof of \cref{blocklemma}:]
	\phantom\qedhere Let $\base$ be a basepoint in $B \subseteq \Sigma$ and $\mathcal M$ be the set of all Morse gauges. We equip $\relMorse{\Sigma}{B}$  with the subspace topology of $ \morse \Sigma=\underset{\underset{N\in \mathcal M}{\longrightarrow}}{\lim}  \morse^{N}\Sigma_{\base}$. We have to show that the map 
		\begin{align*}
	\iota_{*}:\relMorse{\Sigma}{B} &\embed \morse B=\underset{\underset{N\in \mathcal M}{\longrightarrow}}{\lim}  \morse^{N} B_{\base} \\
		\gamma(\infty)&\mapsto \gamma(\infty)
	\end{align*}
	is continuous. 
	We prove the statement in two steps. For any Morse gauge $N \in \mathcal M$, let 	
	\[\morse^N \Sigma_{\base}\cap  \morse B \coloneqq \{ \as \alpha \in \morse^N \Sigma_{\base} \mid \exists \beta \subseteq B \text{ such that } \beta(0)= \base \text{ and } \beta \in \as\alpha\}. \]  
	  We endow $\morse^N \Sigma_{\base}\cap  \morse B$ with the supspace topology of $ \morse^N \Sigma_p$ 
and study the direct limit \begin{align*}
\underset{\underset{\mathcal M}{\longrightarrow}}{\lim} ( \morse^N \Sigma_{\base}\cap  \morse B)
\end{align*} with the induced direct limit topology.
	
	We will show in two steps that the following inclusions are continuous: 
	\setcounter{equation}{0}
	\begin{align}
	\bar\iota:  \relMorse{\Sigma}{B} \embed \underset{\underset{\mathcal M}{\longrightarrow}}{\lim} ( \morse^N \Sigma_{\base}\cap  \morse B)  ,~\gamma(\infty) \mapsto \gamma(\infty)\label{map2}\\
		\iota':  \underset{\underset{\mathcal M}{\longrightarrow}}{\lim} ( \morse^N \Sigma_{\base}\cap  \morse B) \embed \morse{B}   ,~\gamma(\infty) \mapsto \gamma(\infty) \label{map}
	\end{align}

	Since $\iota_* =   \iota' \circ \bar \iota$, this will imply that $\iota_{*}$ is continuous as composition of continuous maps.
	
	\begin{enumerate}
		\item [Step 1:]  
	    	For proving that the inclusion $\bar \iota$ is continuous, it is sufficient to show that $ \morse \Sigma$ and its subspace $\relMorse{\Sigma}{B}$ satisfy the assumptions of \cref{dirlim2}.
		\begin{enumerate}
			\item Since $\morse\Sigma_{\base}$ is $\sigma$-compact, there exists an ascending sequence of natural numbers $N_1 \le N_2\dots$ such that  $\morse\Sigma =\underset{\underset{i\in \N}{\longrightarrow}}{\lim}  \morse^{N_i} \Sigma_{\base}$ by Lemma 2.6 in \cite{Artingroups}. Hence, $\morse \Sigma$ is a direct limit of countably many topological spaces.
			\item By \cref{closedisclosed}, $\relMorse{\Sigma}{B}$ is closed in $\morse \Sigma$,
			\item By \cref{goodembeddings}, the inclusion maps $\iota_{N,N'}:  \morse^N \Sigma_{\base} \embed  \morse^{N'} \Sigma_{\base}$ are topological embeddings for all Morse gauges $N$, $N'$ such that $N \le N'$.
		\end{enumerate}

		\item  [Step 2:] 	For proving that the inclusion $\iota'$ is continuous, it is sufficient to show that all assumptions of \cref{dirlim1} are satisfied.
	
			Let $N$ be a Morse gauge. Since $B$ is a convex subspace of $\Sigma$ and because $B$ contains all geodesic rays in $\Sigma$ that start in $B$ and are asymptotic to a geodesic ray in $B$, \[\iota'(\morse^N \Sigma_{\base}\cap  \morse B) \subseteq \morse^N B_{\base}.\]
		Furthermore, \[\iota'_N: \morse^N \Sigma_{\base}\cap  \morse B \embed \morse^N B_{\base}, ~\gamma(\infty) \mapsto \gamma(\infty)\] is continuous. Indeed, 
		Let $\gamma(\infty) \in  \morse^N \Sigma_{\base}\cap  \morse B$ and $V_n(\iota'_N(\gamma))$ a neighborhood of $\iota'_N(\gamma(\infty)) $ in $\morse^N B_{\base}$ as in \cref{sets}. Note that $\morse^N \Sigma_{\base}$ contains a neighborhood about $\gamma(\infty)$ of the form $V_n(\gamma)$ as well. Since $B$ is a convex subset of $\Sigma$ and because $B$ contains all geodesic rays in $\Sigma$ that start in $B$ and are asymptotic to a geodesic ray in $B$ we have $\iota'_N(V_n(\gamma)\cap \morse B) \subseteq V_n(\iota'_N(\gamma))$. Hence, $\iota'_N$ is continuous. Thus, all assumptions of \cref{dirlim1} are satisfied.\qed\end{enumerate}	 \end{proof}

We finish this section with the following example of Elia Fioravanti showing that the inverse of $\iota_*$ in \cref{blocklemma} need not be continuous.
	\begin{figure}[h]
	\centering
	\includegraphics{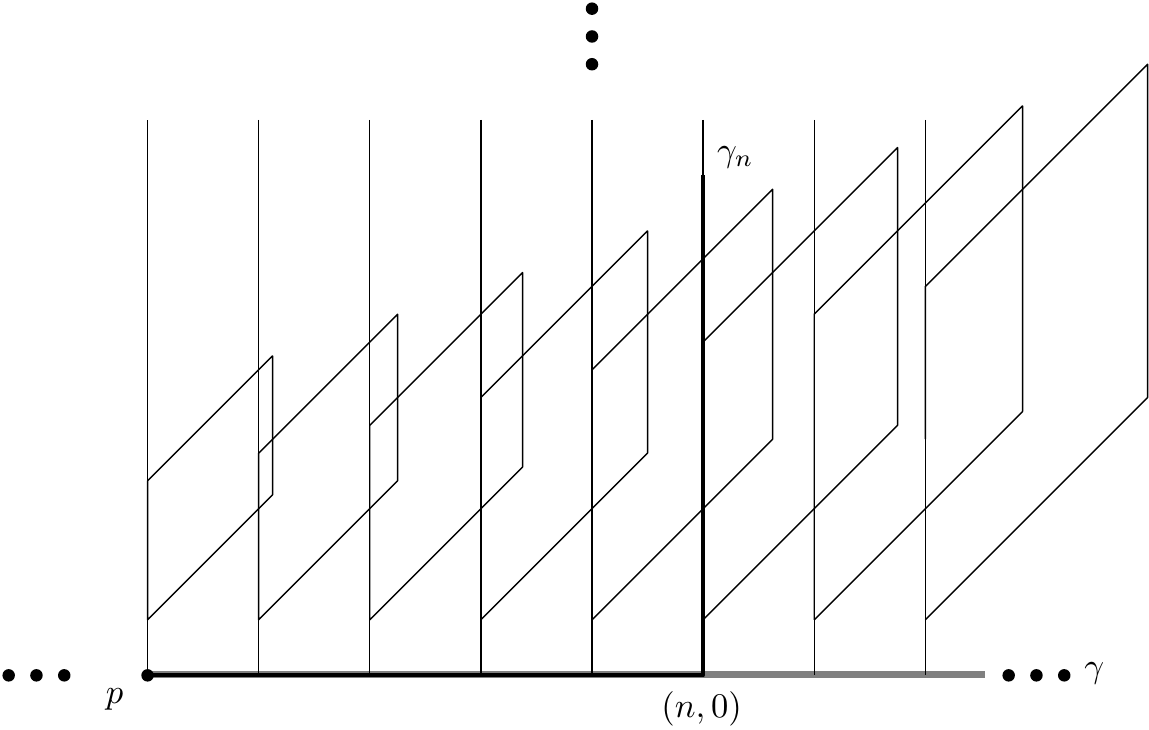}
	\caption{The space $\Sigma$ in \cref{exNotcontinuous}. }
	\label{fig:treeElia}
\end{figure}
\begin{example}[See \cref{fig:treeElia}]
		\label{exNotcontinuous}	
	Let $B$ be the subspace of $\R^2$ that consists of $[0, \infty)$ and all vertical geodesic rays of the form $ \{(n, y) \in \R^2 \mid y \in \R_{\ge 0}\}$ where $n \in \N$. We choose $\base\coloneqq (0,0)$ as basepoint.  Let $ \gamma: \R \to B$,  $\gamma(t) \coloneqq (t,0)$ and  $\gamma_n: \R \to B$, $\gamma_n(t) \coloneqq \gamma(t)$ for all $t \le n$ and $\gamma_n(t)\coloneqq (n,t-n)$ for all $t \ge n$. 
	Then $\gamma$ and $\gamma_n$, $n \in \N$,  are Morse geodesic rays and $\lim_{n \to \infty} \gamma_n(\infty) = \as \gamma$ in $\morse B$. 
	Now, we attach to each geodesic ray $\gamma_n$ a filled square along the geodesic segment connecting $(n,1)$ and $(n,n+1)$ as in the figure. 
	Let $\Sigma$ be the space we obtain this way. The geodesic rays $\gamma_n$, $n \in \N$, and $\gamma$ are still Morse geodesic rays in $\Sigma$. But now, the Morse gauge of $\gamma_n$ grows in $n$. Hence, the sequence $( \gamma_n(\infty))_{n \in \N}$ does not converge to $\as\gamma$ if  we endow $\morse B= \relMorse{\Sigma}{B}$ with the subspace topology of $\morse\Sigma$.   Indeed, for each $N$, the set of $N$-Morse geodesic rays in $\Sigma$  is finite. Hence, $\{ \gamma_n(\infty)\} \cap \morse^N \Sigma_p$ is finite for each Morse gauge and thus, $\{ \gamma_n(\infty)\} $ is a closed subspace of $\morse \Sigma$. In particular, no limit point of  $\{\gamma_n(\infty)\}$ lies outside of this set. 
\end{example}

\section{Proof of \cref{thm2}}
\label{subsec:types}
Let $\Sigma$ be a proper CAT(0) space with treelike block decomposition $\mathcal B$. Let $\Tblock$ be the adjacency graph of $\mathcal B$ and $\base \in \Sigma$ be a basepoint that is not contained in a wall. Now, we apply the cutset property (\cref{cor:diffitinerary2}) and the loneliness property (\cref{Morse_conn_of_type_Ainf}) to study the connected components of $\morse \Sigma$. This will lead to a proof of \cref{thm2}.

\begin{defn}[Itineraries of boundary points]
Let $\gamma(\infty) \in \rand \Sigma$ be a boundary point of $\Sigma$. The \textit{itinerary $I(\gamma(\infty))$} of $\gamma(\infty)$ is the itinerary of the representative of $\gamma(\infty)$ that starts at $\base$.
\end{defn}

\begin{defn}
	\label{def:types}
 We say that a connected component $\kappa$ of $\morse \Sigma$ is of 
	\begin{enumerate}
		\item \textit{type \typeAf} if $\Tblock$ contains a finite path $I$ such that all elements in $\kappa$ have itinerary $I$; 
		\item \textit{type \typeAinf} if $\Tblock$ contains an infinite path $I$ such that all elements in $\kappa$ have itinerary $I$;
		\item \textit{type \typeB} if $\kappa$ contains two elements of distinct itinerary.
	\end{enumerate}
\end{defn}

By definition, every connected component is of exactly one of the types defined above. 

\begin{lemma}
	\label{lem:con_Ainfinite}
	Let $\kappa$ be a connected component of $\morse \Sigma$. The following statements are equivalent: 
	\begin{enumerate}
		\item The connected component $\kappa$ is of type \typeAinf.
		\item If $\gamma(\infty) \in \kappa$, then $I(\gamma)$ is infinite.
		\item If $\gamma$ is a geodesic ray starting at $\base$ that represents an element in $\kappa$, then $\gamma$ does not end in a block. 
	\end{enumerate}
\end{lemma}

\begin{proof} We show $(1) \Rightarrow (2)\Rightarrow(3) \Rightarrow(1)$.
	
	$(1) \Rightarrow (2)$ If $\kappa$ is of type \typeAinf, then $\Tblock$ contains an infinite path $I$ such that all elements in $\kappa$ have itinerary $I$. In particular, each element in $\kappa$ has infinite itinerary.
	
	 $(2) \Rightarrow (3)$ Suppose that each element in $\kappa$ has infinite itinerary. Let $\gamma$ be a geodesic ray starting at $\base$ that represents an element in $\kappa$. Then $\gamma$ does not end in a block as otherwise, $I(\gamma)= I(\gamma(\infty))$ would be a finite path.
	 	
	$(3) \Rightarrow (1)$: Suppose that every geodesic ray starting at $\base$ that represents an element in $\kappa$ does not end in a block. Then each element in $\kappa$ has infinite itinerary by the definition of itineraries of geodesic rays. It remains to show that $\kappa$ does not contain two elements with distinct itineraries. 
	Assume for a contradiction that there are two distinct geodesic rays $\alpha$ and $\beta$ starting at $\base$ such that $\alpha(\infty)$, $\beta(\infty) \in \kappa$ and $I(\alpha)\neq I(\beta)$. Then there is a wall $W$ that appears in one of both itineraries $I(\alpha)$ and $I(\beta)$ but not in both. By~\cref{cor:diffitinerary2}, there is a geodesic ray $\gamma \subseteq W$ such that $\as \gamma \in \kappa$. Let $\gamma_{\base}$ be the unique geodesic ray starting at $\base$ that is asymptotic to $\gamma$. By \cref{lem:asymptoticrays}, $I(\gamma(\infty))$ is a finite path. Thus, $\kappa$ contains at least one element with finite itinerary. By definition of itineraries, such a geodesic ray ends in a block-- a contradiction.\end{proof}

We are now able to prove the following theorem which directly implies \cref{thm2}. Since  \cref{cor:firstgeneralization} is a direct consequence of \cref{thm2} and \cref{corOpenset} in \cref{subsecMorse3}, the following proof will also complete the proof of \cref{cor:firstgeneralization}.
\begin{thm}\label{thm3}
	Let $\kappa$ be a connected component of $\morse \Sigma$. \begin{enumerate}
		\item If $\kappa$ is of type \typeAf, then there exists a block $B$ so that the representative of every point in $\kappa$ starting at $\base$ ends in a block $B$. Moreover, $\kappa$ is homeomorphic to a connected component of $\relMorse{\Sigma}{B}$ endowed with the subspace topology of $\morse \Sigma$.
		\item If $\kappa$ is of type \typeAinf, then $|\kappa| = 1$. 
		\item If $\kappa$ is of type \typeB,~then it contains an equivalence class of a geodesic ray  in a wall.
	\end{enumerate}
\end{thm}

\cref{fig:conn_components_types} summarizes the classification of connected components in \cref{thm3}. 
\begin{figure}[h]
	\centering
	\includegraphics[]{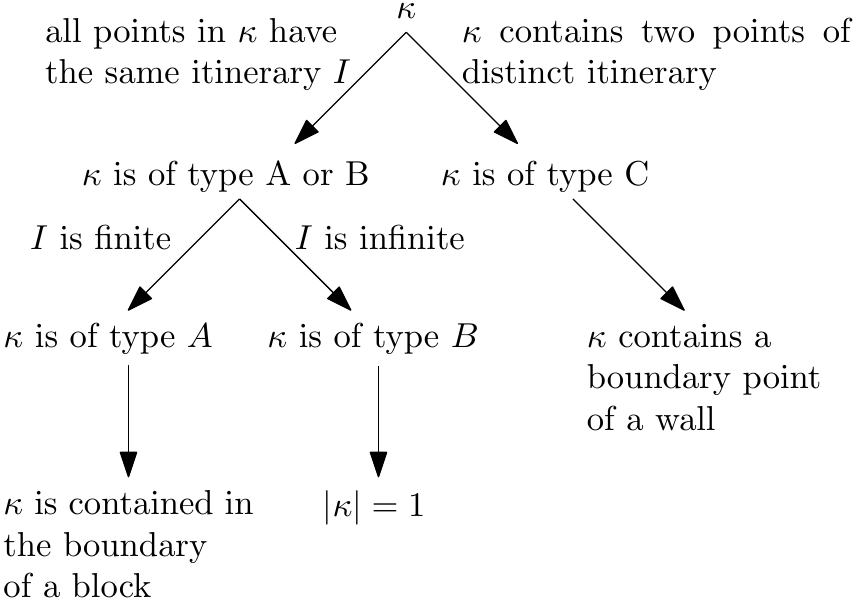}
	\caption{Possible types of a connected component $\kappa$ of $\morse \Sigma$. 
		The arrows denote implications valid under the conditions of the labels at the arrows.}
	\label{fig:conn_components_types}
\end{figure}

\begin{proof}[Proof of \cref{thm3}]	Let $\kappa$ be a connected component of $\morse \Sigma$.
\phantom\qedhere
\begin{enumerate}	
	\item Suppose that $\kappa$ is of type \typeAf. Then there exists a finite path $I$ in $\mathcal T$ such that  all elements in $\kappa$ have itinerary $I$; Then there exists a block $B$ such that each geodesic ray starting at $p$ and representing an element in $\kappa$ ends in $B$. In particular, every equivalence class in $\kappa$ has a representative that is contained in $B$, i.e. $\kappa \subseteq \relMorse{\Sigma}{B}$. Thus, $\kappa$ is homeomorphic to a connected component of $\relMorse{\Sigma}{B}$  equipped with the subspace topology of $\morse \Sigma$.
	
	\item Suppose that $\kappa$ is of type \typeAinf. Then there exists an infinite path $I$ in $\mathcal T$ such that $I(\gamma)=I$ for each geodesic ray $\gamma$ starting at $\base$ such that $\gamma(\infty) \in \kappa$. By the loneliness property (\cref{Morse_conn_of_type_Ainf}), all geodesic rays with itinerary $I$ are asymptotic. Hence, $\kappa$ contains only one element. 	
	\item Suppose that $\kappa$ is of type \typeB. Then $\kappa$ contains two points with different itineraries $I_1$ and $I_2$. According to the cutset property (\cref{cor:diffitinerary2}), $\kappa$ contains the equivalence class of a geodesic ray that is contained in a wall.\qed 
\end{enumerate} \end{proof}

\begin{remark}
	The content of this section is with respect to the direct-limit topology of the Morse boundary. This does not imply the validity of the statements above for the visual boundary. However, the methods of the proofs can be transferred to (path) components of the visual boundary. The arguments used above can be repeated to confirm the following statements.
	\begin{cor}
		If $\kappa$ is a connected component of $\vis\Sigma$ of type \typeAinf~ and contains a boundary point that is Morse, then $|\kappa| = 1$.
	\end{cor}
The next corollary is related to \cref{lem:con_Ainfinite} and Lemma 7 in section 1.7 of \cite{CrokeKleiner}
		\begin{cor}
	 If a curve $c:[a,b]\to\vis\Sigma$ starts at a point $c(0)$ of infinite itinerary $I$, then $I(c(t)) = I$ for all $t \in [a,b]$ or there is a time $t$ such that $c(t)$ has finite itinerary.
	\end{cor}
\end{remark}

\section{Applications to RACGs}
\label{sec:mainproofs}
In \cref{subsec1spaces}, we will determine which totally disconnected topological spaces occur as Morse boundaries of RACGs by applying Theorem 1.4 in \cite{Artingroups}.  
In \cref{subsecblockdec}, we will complete the proof of \cref{thm1} by showing that each Davis complex of infinite diameter has a non-trivial treelike block decomposition. In \cref{sec:ex}, we will introduce a new graph class $\mathcal C$ consisting of graphs that correspond to RACGs with totally disconnected Morse boundaries. We will investigate this graph class by studying some examples of the literature. 
\subsection{Totally disconnected spaces arising as Morse boundaries of RACGs}
\label{subsec1spaces}

In this subsection, we study which totally disconnected topological spaces arise as Morse boundaries of RACGs. We will use Theorem 1.4 in \cite{Artingroups} and I would like to thank  Matthew Cordes for his comment how to apply this theorem to RACGs.

 An \textit{$\omega$-Cantor space}  is the direct limit of a sequence of Cantor spaces $C_1 \subset C_2\subset C_3 \dots$ such that $C_i$ has empty interior in $C_{i+1}$ for all $i \in \N$. By \cite[Thm 3.3]{Artingroups}, any two $\omega$-Cantor spaces are homeomorphic.  
We will apply the following Theorem \cite[Thm 1.4]{Artingroups}.
\begin{thm}[Charney--Cordes--Sisto]
	\label{thm1.4}
	Let $G$ be a finitely generated group. Suppose that $\morse G$ is totally disconnected, $\sigma$-compact, and contains a Cantor subspace. Then $\morse G$ is either a Cantor space or an $\omega$-Cantor space. It is a Cantor space if and only if $G$ is hyperbolic, in which case $G$ is virtually free. 
\end{thm}

A \textit{suspension} of a graph $\Delta$ is a join of a graph consisting of two vertices and  $\Delta$. 
\begin{cor}
	Let $W_\Delta$ be a RACG with defining graph $\Delta$ whose Morse boundary is totally disconnected. 
 If $\Delta$ is a clique or a non-trivial join, then  $\morse W_\Delta$ is empty.
 If $\Delta$ consists of two non-adjacent vertices or is a suspension of a clique, then  $\morse W_\Delta$ is virtually cyclic and $\morse W_\Delta$ consists of two points. Otherwise, if $\Delta$ does not contain any induced $4$-cycle then $\morse W_\Delta$ is a Cantor space. 
In the remaining case, $\morse W_\Delta$ is an $\omega$-Cantor space.

\end{cor}

\begin{proof}
	If the graph $\Delta$ is a clique or a non-trivial join, then $\morse W_\Delta$ is empty by \cref{lem:con_empty}. 	 Otherwise, $W_\Delta$ contains a rank-one isometry by Corollary B in \cite{Sageev_Caprace}. If $\Delta$ consists of two non-adjacent vertices, then $W_\Delta$ is isomorphic to the infinite Dihedral group. If $\Delta$ is a suspension of a clique, then $W_\Delta$ is the direct product of the infinite Dihedral group with a finite  right-angled Coxeter group. In both cases, $W_\Delta$ is quasi-isometric to $\Z$ and the Morse boundary of $W_\Delta$ consists of two points.
	
	In the remaining case, an induction on the number of vertices shows that $\Delta$ contains  either the graph consisting of three pairwise non-adjacent vertices as induced subgraph or the graph consisting of an edge and a further single vertex. These graphs correspond to special subgroups of $W_\Delta$ that are quasi-isometric to a tree whose vertices have a degree of at least three. Thus, in the remaining case,  $W_\Delta$ is not quasi-isometric to $\Z$. In particular, $W_ \Delta$ is not virtually cyclic. 
	
	It remains the case where $W_\Delta$ is not virtually cyclic and contains a rank-one isometry.  We will show that the assumptions of \cref{thm1.4} are satisfied. If $\Delta$ does not contain an induced $4$-cycle, $W_\Delta$ is hyperbolic (See\cite[Thm 12.2.1, Cor 12.6.3]{Davis}, \cite{Gromov}).  Then \cref{thm1.4} will imply that  $\morse W_\Delta$ is a Cantor space. Otherwise, if $\Delta$ contains an induced $4$-cycle, \cref{thm1.4} will imply that $\morse W_\Delta$ is an $\omega$-Cantor space.	
	
	For proving the assumptions of \cref{thm1.4}, we have to show that $\morse W_\Delta$ is $\sigma$-compact and contains a Cantor space as subspace. The $\sigma$-compactness follows from Proposition 3.6 in \cite{CharSul}. The existence of the Cantor subspace can be proven similarly as in Lemma 4.5 in \cite{Artingroups}: Osin \cite{Osin16} concluded from \cite{Sisto18} that if a group acts properly on a proper CAT(0) space and contains a rank-one isometry, then the group is either virtually cyclic or acylindrically hyperbolic. As $W_\Delta$ is not virtually cyclic, $W_\Delta$ is acylindrically hyperbolic. Thus, Theorem 6.8 and Theorem 6.14 of \cite{DahmaniandCo17} imply the existence of a hyperbolically embedded free group. By \cite{Sisto16}, this free subgroup is quasi-convex and thus stable. The Morse boundary of this stable subgroup is an embedded Cantor space in the Morse boundary of the ambient group. 	
\end{proof}

\subsection{Block decompositions of Davis complexes}
\label{subsecblockdec}

In this subsection, we prove that each Davis complex of infinite diameter has a non-trivial treelike block decomposition. This will complete the proof of \cref{thm1}. 

Recall that the RACG associated to a finite, simplicial graph $\Delta=(V,E)$ is the group 
\[W_\Delta =\langle V\mid v^2=1~\forall v \in V, uv=vu ~\forall ~\{u,v\} \in E \rangle.\]
Let $\Cayley{W_\Graph}{V}$ be the Cayley graph of $W_\Graph$ with respect to the generating set $V$. 
The \textit{Davis complex} $\davislambda$ of $W_\Graph$ is the unique CAT(0) cube complex with $\Cayley{W_\Graph}{V(\Graph)}$ as $1$-skeleton so that 
\begin{itemize}
	\item each $1$-skeleton of a cube in $\davislambda$ is an induced subgraph of $\Cayley{W_\Graph}{V(\Graph)}$ and 
	\item each set of vertices in $\Cayley{W_\Graph}{V(\Graph)}$ that spans an Euclidean cube is the $1$-skeleton of an Euclidean cube in $\davislambda$. 
\end{itemize}
See~\cite[p.9-14]{Davis} for more details and \cite{Sageev,Haglund2008,SageevNotes} for general information of CAT(0) cube complexes.

\begin{defn} Let $\Graph$ be a finite simplicial graph. 
	A subgroup of a RACG $\W{\Graph}$ is \textit{special} if it has an induced subgraph of $\Graph$ as defining graph. 
\end{defn}
\begin{remark} 
	\label{rem_Davis}
	\begin{enumerate}
		\item The trivial graph $(\emptyset, \emptyset)$ is an induced subgraph of $\Graph$. The Davis complex of $(\emptyset, \emptyset)$ consists of a vertex corresponding to the identity of the trivial group.
		\item 	If $\Graph'$ is an induced subgraph of a graph $\Graph$, then $\Cayley{W_{\Graph'}}{V(\Graph')}$ is an induced subgraph of $\Cayley{W_\Graph}{V(\Graph)}$. This extends to a unique isometric embedding $\Sigma_{\Graph'} \embed \Sigma_{\Delta}$ which we call the \textit{canonical embedding} of $\Sigma_{\Graph'}$ into $\Sigma_{\Graph}$. In the following, we identify $\Sigma_{\Graph'}$ with its image. 
		\item Since $g \in W_{\Graph}$ acts by cubical automorphisms on $\Sigma_{\Delta}$, $g \cdot \Sigma_{\Delta'}$ is also an embedded copy of $\Sigma_{\Delta'}$ in $\Sigma_{\Delta}$. 
		\item 	Let $\bar \Graph$, $\Graph'$ be two induced subgraphs of $\Graph$ and $g$, $h \in W_{\Graph}$. Then $gW_{\bar \Graph}\subseteq hW_{\Graph'}$ $($resp.~$gW_{\bar\Graph} = hW_{\Graph'})$ if and only if $\inv{g}h \in W_{\Graph'}$ and $\bar \Graph \subseteq \Graph'$ $($resp. $\bar \Graph = \Graph')$. See \cite[Thm 4.1.6]{Davis}. \label{stabilizer}
	\end{enumerate}

\end{remark}

\begin{prop}
	\label{Prop}
	Let $\Delta$ be a finite, simplicial graph that can be decomposed into two distinct proper induced subgraphs $\Delta_1$ and $\Delta_2$ with the intersection graph $\Lambda=\Delta_1 \cap \Delta_2$. Let $\Sigma_{\graphzero}$, $\Sigma_{\graphone}$, $\Sigma_{\Lambda}$ be the canonically embedded Davis complexes of $\graphzero$, $\graphone$ and $\Lambda$ in $\Sigma_{\Graph}$.
	The collection \[\mathcal B \coloneqq \{g\Sigma_{\graphzero}\mid g \in W_{\Graph}\}\cup \{g\Sigma_{\graphone}\mid g \in W_{\Graph}\}\]
	is a treelike block decomposition of $\Sigma_{\Delta}$. 
	The collection of walls $\mathcal W$ is given by \[\mathcal W = \{g\Sigma_{\Lambda}\mid g \in W_{\Graph}\}.\]
\end{prop}

\begin{proof}	
 By \cref{criteriontreelike}, it suffices to show that 
	\begin{enumerate}
		\item $\Sigma= \bigcup_{B \in \mathcal B} B$ (covering condition);
		\item every block has a parity $(+)$ or $(-)$ such that two blocks intersect only if they have opposite parity (parity condition);
		\item there is an $\epsilon >0$ such that two blocks intersect if and only if their $\epsilon$-neighborhoods intersect ($\epsilon$-condition).
	\end{enumerate}

	\textbf{(1) Covering condition:}
		By Proposition 8.8.1 of Davis in \cite{Davis}, $W_{\Graph} = W_{\graphzero}*_{W_{\Lambda}} W_{\graphone}$. 
		Hence, the cosets $\{g W_{\graphzero} \mid g \in W_{\Graph}\}$ and $\{g W_{\graphone}\mid g \in W_{\Graph}\}$ cover $W_{\Graph}$. 
		As the one-skeleton of $ \Sigma_{\Graph}$ is the Cayley graph of $W_\Graph$, $\mathcal B$ covers $\Sigma_{\Graph}$. 
		
		\textbf{(2) Parity condition:}		
			We give parity (-) to each block of the form $g \daviszero$ and parity $(+)$ to each block of the form $g \davisone$. Let $i \in \{1,2\}$ and $B_1\coloneqq g_1 \Sigma_{\Delta_i}$ and $B_2\coloneqq g_2 \Sigma_{\Delta_i}$ two distinct blocks of the same parity. By definition of the Davis complex, the $0$-skeletons of $B_1$ and $B_2$ are the coset $g_1W_{\Delta_i}$ and $g_2W_{\Delta_i}$. As $B_1$ and $B_2$ are two distinct blocks, $g_1W_{\Delta_i}\neq g_2W_{\Delta_i}$. Hence, $g_1 W_{\Delta_i}\cap g_2 W_{\Delta_j}= \emptyset$. Thus, the $0$-skeletons of $B_1$ and $B_2$ don't intersect. This implies that $B_1 \cap B_2 = \emptyset$. Indeed, by \cref{rem_Davis}, the blocks $B_1$ and $B_2$ are isometric to the Davis complex of $W_{\Graph_i}$. Hence, if a set of vertices in $\Cayley{W_{\Graph_i}}{V(\Graph_i)}$ spans an Euclidean cube, then it spans a cube in $B_j$, $j \in \{1,2\}$ and on the other hand, each $1$-skeleton of a cube in $B_j$ is an induced subgraph of $\Cayley{W_{\Graph_i}}{V({\Graph_i})}$. Thus, $B_1$ and $B_2$ intersect if and only if their $0$-skeletons intersect.
		 
		 \textbf{(3) $\epsilon$-condition:} The proof concerns hyperplanes of CAT(0) cube complexes, see~\cite{Sageev,Haglund2008,SageevNotes} for the definition and basic properties. Let $\epsilon \in (0,\frac{1}{2})$. 
		 Assume that the $\epsilon$-neighborhoods of two blocks $B_1$, $B_2 \in \mathcal B$ intersect. We have to show that $B_1 \cap B_2 \neq \emptyset$.
		 		 		 
		We observe that there is no hyperplane $H$ that separates $B_1$ and $B_2$. In other words: If we delete a hyperplane $H$ outside of $B_1 \cup B_2$, then $B_1$ and $ B_2$ lie in a common connected component of the resulting space. Indeed otherwise, each geodesic segment connecting $B_1$ and $ B_2$ has to pass through $H$. As hyperplanes are equivalence classes consisting of midcubes, $d(x,y)\ge \frac{1}{2}$ for all $x \in B_1 \cup B_2$, $y \in H$. As $\epsilon < \frac{1}{2}$, this implies that the $\epsilon$-neighborhoods of $B_1$ and $ B_2$ don't intersect -- a contradiction. We conclude that there is no hyperplane $H$ outside of $B_1 \cup B_2$ that separates $B_1$ and $B_2$. Thus, the distance of the $1$-skeletons of $B_1$ and $ B_2$ is zero by Theorem 4.13 in~\cite{Sageev}. In particular, $B_1$ and $ B_2$ intersect.
		 
		 	 It remains to show that each wall is of the form $g \Sigma_{\Lambda}$, $g \in W_{\Graph}$. Indeed, 	 let $B_1\coloneqq g_1 \Sigma_{\Delta_1}$ and $B_2\coloneqq g_2 \Sigma_{\Delta_2}$ be two distinct blocks of distinct parity that have non-empty intersection. We see as in $(2)$ that the $(0)$-skeletons of $B_1$ and $B_2$ are the left-cosets $g_1 W_{\Delta_1}$ and $g_2 W_{\Delta_1}$. Since we can write $W_\Graph$ as an amalgamated free product $W_{\graphzero}*_{W_{\Lambda}} W_{\graphone}$, the intersection $g_1 W_{\Delta_1} \cap g_2 W_{\Delta_1}$ is a left coset of $W_{\Lambda}$, and the cube complex $C$ spanned by the vertices in this left coset is contained in $B_1 \cap B_2$. Now, $B_1 \cap B_2$ does not contain any other point in $B_1\cup B_2$ because each cube in $B_j$, $j \in \{1,2\}$ is spanned by vertices that are contained in $g_j W_j$. Thus, every point $x \in B_1\cup B_2 \setminus B_1\cap B_2$ lies in a cube that is contained in at most one of the two blocks $B_1$ and $B_2$.\end{proof}
As a result, \cref{thm1} is a direct consequence of \cref{thm2}, \cref{Prop} and \cref{lem:con_empty}.

\subsection{Examples}
\label{sec:ex}

\cref{thm1} can be used to construct new examples of RACGs with totally disconnected Morse boundaries. First, we generalize the example of Charney--Sultan studied in Section 4.2 of \cite{CharSul}. Afterwards, we introduce a class $\mathcal C$ of graphs corresponding to RACGs with totally disconnected Morse boundaries. Finally, we  study interesting examples that lie in $\mathcal C$. I would like to thank Ivan Levcovitz, Jacob Russell and Hung Cong Tran for their comments on these examples. 

\begin{defn}\label{CharneySultangraph}
	A finite, connected graph $\Delta$ is called a \textit{Charney-Sultan-graph} if 
	 $\Delta$  is the union of two distinct proper induced subgraphs $C$ and $J$ so that $C$ is a cycle of length at least $5$ and $J$ is a non-trivial join of two induced subgraphs of $\Delta$.
\end{defn} 

\begin{lemma} \label{lem:charsultgraph}
	If $\Delta$ is a Charney-Sultan-graph, then $\morse W_{\Delta}$ is totally disconnected.
\end{lemma}

		\begin{proof}
		Suppose that $\Delta$ is a Charney-Sultan graph. Then $\Delta$ is the union of a cycle $C$ and a non-trivial join $J$. If $J$ contains all the vertices of $C$, $\Delta$ coincides with $J$. Then $W_\Delta$ is the direct product of two RACGs and $\morse W_\Delta$ is empty. 
		
		If $J$ does not contain all vertices of $C$, $C$ contains a path $P$ of length at least two that connects two vertices in $J$ so that no inner vertex of $P$ is contained in $J$. Let $\Delta'$ be the graph that we obtain by removing the inner vertices of $P$ from $\Delta$. Then $\Delta = \Delta'\cup P$ and $\Delta' \cap P$ consists of the two non-adjacent end vertices of $P$. If $\Sigma_{\Delta'}$ has totally disconnected Morse boundary, then $W_\Delta$ has totally disconnected Morse boundary by 	\cref{Cor1}.
				
		Let $C'$ be the subgraph of $C$ that we obtain by removing the inner vertices of $P$. Then $\Delta' = C' \cup J$. If $C' \cap J = J$, $\Delta'$ is a non-trivial join. Then  $W_{\Delta'}$ is the direct product of two infinite CAT(0) groups and thus, $\morse \Sigma_{\Delta'}$ is empty. Otherwise, $C' \cap J$ is a proper subgraph of $J$.  Since the graph $C'$ is a path, $W_{C'}$ is quasi-isometric to a free group. Thus, $\morse \Sigma_{C'}$ is totally disconnected. In particular, $\relMorse{\Sigma_{\Delta'}}{\Sigma_{C'}}$ endowed with the subspace topology of $\morse \Sigma_{C'}$ is totally disconnected. As $J$ is a non-trivial join, $W_J$ is the direct product of two infinite CAT(0) groups and thus, $\morse \Sigma_J$ is empty. By \cref{Cor1}, $\Sigma_{\Delta'}$ has totally disconnected Morse boundary.\end{proof}

We enlarge the class of Carney-Sultan-graphs in the following manner. 
\begin{defn}[The class $\mathcal C$ of \textit{clique-square-decomposable} graphs] {\color{white}{asdgtagdsaadfa\\}}
	\label{def:graphclass}
	Let $\graphclass$ be the smallest class of finite graphs such that 
	
	\begin{enumerate}
		\item each finite graph without edges is contained in $\graphclass$; \label{itemgraphclass0}
		\item each finite tree is contained in $\graphclass$; \label{itemgraphclass1}
		\item each Charney-Sultan graph is contained in $\graphclass$; \label{itemgraphclass2}
		\item each clique is contained in $\graphclass$; \label{itemgraphclass2.1}
		\item each non-trivial join of two graphs is contained in $\graphclass$; \label{itemgraphclass2.2}
		
		\item \label{itemgraphclass4} the union of two graphs $\Lambda_1$, $\Lambda_2 \in \graphclass$ is contained in $\graphclass$ if $\Lambda_1\cap \Lambda_2$ is an induced subgraph of $\Lambda_1\cup \Lambda_2$ so that one of the following three conditions is satisfied:
		\begin{itemize}
			\item $\Lambda_1\cap \Lambda_2$ is empty.
			\item $\Lambda_1 \cap \Lambda_2$ is a clique.
			\item $\Lambda_1 \cap \Lambda_2$ is contained in a non-trivial join of two induced subgraphs of $\Lambda_1\cup\Lambda_2$.
		\end{itemize}
	\end{enumerate}
\end{defn}

Let $\Delta$ be a graph in $\mathcal C$. If $\Delta$ satisfies (1) or (2) then $W_\Delta$ is quasi-isometric to $\Z / 2 \Z$, $\Z$ or to a free group of rank at least two and $\morse W_{\Delta}$ is totally disconnected. If $\Delta$ satisfies (3), $\morse W_{\Delta}$ is totally disconnected by \cref{lem:charsultgraph}. If $\Delta$ satisfies (4), $W_{\Delta}$ is a finite group and $\morse W_\Delta = \emptyset$. If $\Delta$ satisfies (5), $W_{\Delta}$ is the direct product of two infinite RACGs. In this case, each geodesic ray in $\Sigma_\Delta$ is contained in an Euclidean flat and thus, $\morse W_{\Delta} = \emptyset$. 
If $\Delta$ satisfies (6), we can decompose $\Graph$ into two graphs $\Lambda_1$, $\Lambda_2 \in \mathcal C$ satisfying the properties listed in the definition above. 
If $\Lambda_1$ and $\Lambda_2$ satisfy (1), (2), (3), (4) or (5), then \cref{Cor1} implies that $\morse W_{\Graph}$ is totally disconnected. Otherwise, we repeat the same argumentation for $\Lambda_1$ and $\Lambda_2$. As $\Graph$ is a finite graph, this algorithm ends after finitely many step. Thus, by applying \cref{Cor1} several times we obtain 
\begin{cor*}[\ref{cor:C-class}]
	If $\Delta \in \mathcal C$, then $\morse W_\Delta$ is totally disconnected. 
\end{cor*}

Next, we examine the graph class $\mathcal C$ by studying so-called $\mathcal{CFS}$ graphs. The \textit{four-cycle graph $\Graph^4$} of a graph $\Graph$ is a graph whose vertices are the induced cycles of length four. 
Two vertices of $\Graph^4$ are connected by an edge if the corresponding $4$-cycles have a pair of vertices in common that are not adjacent in $\Graph$. The \textit{support} of a subgraph $K$ of $\Graph^4$ is the set of vertices of $\Graph$ that are contained in a $4$-cycle corresponding to a vertex of $K$. 
The following is a generalization in \cite{Behrstock_random} of the original definition of Dani--Thomas\cite{Dani_diver}.

\begin{defn}[$\mathcal{CFS}$]
	\label{def:cFS}
	A graph $\Graph$ is $\mathcal{CFS}$ if it is a join of two graphs $\Delta$ and $K$ where $\Delta$ is a non-trivial subgraph of $\Graph$ and $K$ is a clique (it is allowed that this clique is trivial, i.e., $(\emptyset, \emptyset)$) so that $\Delta^4$ has a connected component whose support coincides with the vertex set $V(\Delta)$ of $\Delta$. 
\end{defn}

Intuitively, a $\mathcal{CFS}$ graph $\Delta$ contains a lot of induced $4$-cycles. One could expect that $W_\Delta$ has totally disconnected Morse boundary. However, an example of Behrstock~\cite{Behrstock} shows that this is wrong in general. On the other hand, Nguyen--Tran \cite{Tran_coars_geom} used an approach similar to this paper for proving that the Morse boundary of $W_\Delta$ is totally disconnected if $\Delta$ is contained in the following graph class $\mathcal{CFS}_0 \subseteq \mathcal{CFS}$.
	\begin{defn}
		\label{def:CFS0}
		Let $\mathcal{CFS}_0$ be the class of all graphs that are $\mathcal{CFS}$, connected, triangle-free, planar, having at least $5$ vertices and no separating vertices or edges.
	\end{defn}
 	On page $3$ in \cite{Tran}, Nguyen--Tran conclude the following result from their considerations:
	\begin{thm}[Nguyen--Tran]
		If $\Delta \in \mathcal{CFS}_0$, then $\morse W_\Delta$ is totally disconnected.
	\end{thm}
It turns out that our theorem includes their result in view of the following proposition: 
 	\begin{prop}
 	 $\mathcal{CFS}_0 \subseteq \mathcal C$.
 \end{prop}
\begin{proof}
	In Proposition 3.11 of \cite{Tran_coars_geom}, Nguyen--Tran prove that each $\Delta \in \mathcal{CFS}_0$ decomposes as a tree of graphs $\mathcal G$ so that each vertex of $\mathcal G$ corresponds to a non-trivial join of two graphs consisting of two and three vertices respectively and so that each edge of $\mathcal G$ corresponds to an induced $4$-cycle in $\Delta$. Thus, $\Delta$ can be obtained by starting with a non-trivial join and adding finitely many non-trivial joins to it. 
	Unlike Nguyen--Tran, we allow to add not only non-trivial joins and cliques but also other more sophisticated graphs. Hence, $\mathcal{CFS}_0 \subseteq \mathcal C$.\end{proof}

 The class $\mathcal C$ is substantially larger than the class $\mathcal{CFS}_0$. The following lemma shows that $\mathcal C$ contains graphs corresponding to RACGs with polynomial divergence of arbitrary high degree. 
  In contrast, each graph in $\mathcal{CFS}_0 \subseteq \mathcal{CFS}$ corresponds to a RACG of quadratic divergence. Indeed, Dani--Thomas~\cite{Dani_diver} proved that a triangle-free graph $\Graph$ is $\mathcal{CFS}$ if and only if $W_\Graph$ has quadratic divergence, and Levcovitz~\cite[Thm 7.4]{Ivan} proved this statement for general graphs .

 
\begin{figure}[h]
	\centering
	\includegraphics{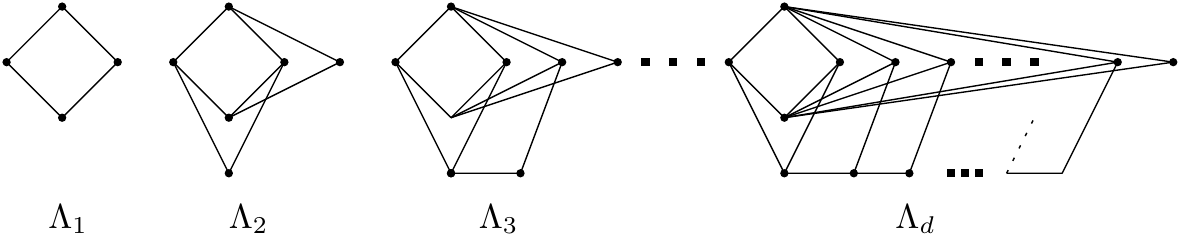}	
	\caption{Dani-Thomas  \cite[Thm 1.2, Sec. 5]{Dani_diver} proved that the pictured graphs $\Lambda_i$, $i \in \N$ correspond to RACGs with polinomial divergence of degree $i$.}
	\label{fig:exDaniThomas}
\end{figure}
\begin{lemma}
	\label{divlem}
	For every $d \in \N$, the graph class $\mathcal C$ contains a graph associated to a RACG whose divergence is of polinomial degree $d$. 
\end{lemma}
\begin{figure}[h]
	\centering
	\includegraphics{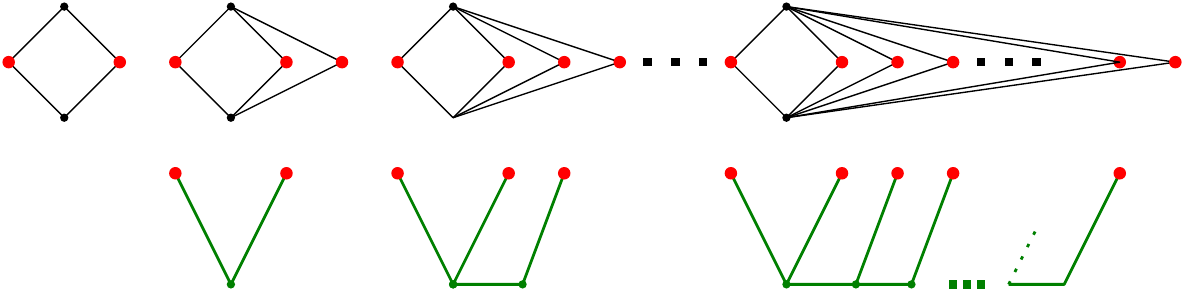}	
	\caption{Decomposition of the graphs in \cref{fig:exDaniThomas} showing that each graph $\Lambda_i$ is contained in $\mathcal C$.  }
	\label{fig:exDaniThomasdecomp}
\end{figure}
\begin{proof}
	Dani-Thomas proved \cite{Dani_diver} that the graphs $\Lambda_i$ in \cref{fig:exDaniThomas} correspond to RACGs with polinomial divergence of degree $i$, $i \in \mathbb N$. Each graph $\Lambda_i$ can be decomposed into a nontrivial join and a tree as shown in \cref{fig:exDaniThomasdecomp}. Each graph in the upper row of \cref{fig:exDaniThomasdecomp} is a nontrivial join and each graph in the lower row in \cref{fig:exDaniThomasdecomp} is a tree. Thus, $\Lambda_i \in \mathcal C$ for all $i \in \mathbb N$. 
\end{proof}

\begin{remark}
	The graphs $\Delta_i$, $i \ge 3$, in \cref{fig:exDaniThomas} correspond to RACGs with totally disconnected Morse boundaries that are not quasi-isometric to any RAAG. Indeed, every right-angled Artin group has either linear or quadratic divergence, as remarked by Behrstock~\cite{Behrstock}. Hence, if a graph is not $\mathcal CFS$ then the corresponding RACG is not  quasi-isometric to a RAAG. 
\end{remark}

Another interesting example is the graph pictured in \cref{fig:exBenZvi} that was studied by Ben-Zvi. 
\begin{lemma}
The graph $\Delta$ pictured in \cref{fig:exBenZvi} is contained in $ \mathcal C \setminus \mathcal{CFS}_0$.
\end{lemma}
	\begin{figure}[h]
	\centering
	\includegraphics{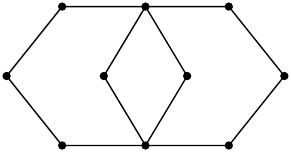}	
	\caption{The defining graph of a RACG studied by Ben-Zvi in \cite[Ex. 2.3]{BenzviFlats}. This graph is contained in $\mathcal C \setminus \mathcal{CFS}_0$.}
	\label{fig:exBenZvi}
\end{figure}
\begin{proof}Let $\Delta_1$ and $\Delta_2$ be the two subgraphs of $\Delta$ pictured in \cref{fig:exBenZvi2}. 
		\begin{figure}[h]
		\centering
		\includegraphics{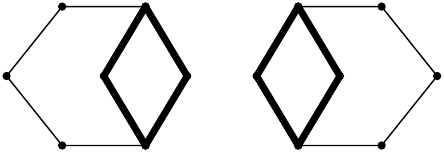}	
		\caption{A decomposition of the graph pictured in \cref{fig:exBenZvi}. The common intersection graph is a $4$-cycle consisting of the bold edges. }
		\label{fig:exBenZvi2}
	\end{figure} Then the intersection $\Lambda$ of $\Delta_1$ and $\Delta_2$ is a $4$-cycle, marked bold in the figure. The subgraphs $\Delta_1$ and $\Delta_2$ are Charney-Sultan graph. The intersection graph $\Lambda$ is an induced $4$-cycle of $\Delta$. In particular, it is contained in a non-trivial join. Hence, $\Delta$ is contained in $\mathcal C$.

It remains to show that $\Delta \notin \mathcal{CFS}_0$.
	The graph $\Delta$ contains one $4$-cycle only and $\Delta^4$ does not have a connected component whose support coincides with the vertex set $V(\Delta)$ of $\Delta$. Hence, $\Delta \notin \mathcal{CFS}$ and as $\mathcal{CFS}_0\subseteq \mathcal{CFS}$, $\Delta \notin \mathcal{CFS}_0$.\end{proof}
\begin{remark}
Ben-Zvi argues that $W_\Delta$ is a CAT(0) group with isolated flats and proves that the visual boundary of the corresponding Davis complex is path-connected. Among other things, the graph $\Delta$ was studied by Ben-Zvi for the following reason: 
	\begin{figure}[h]
	\centering
	\includegraphics{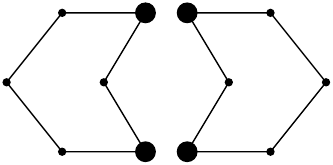}	
	\caption{A decomposition of the graph pictured in \cref{fig:exBenZvi}. The common intersection graph consists of the two vertices which are drawn bold in both pictured graphs.}
	\label{fig:exBenZvi3}
\end{figure}
Let $\Delta_1$ and $\Delta_2$ the two graphs pictured in \cref{fig:exBenZvi3} and $\Lambda = \Delta_1 \cap \Delta_2$. The graph $\Delta$ consists of two vertices. See \cref{fig:exBenZvi3}.
 Ben-Zvi observes that the virtual $\Z^2$ corresponding to the four-cycle in the middle is hidden if we write $W_\Delta$ as $W_{\Delta} = W_{\Delta_1}*_{W_{\Lambda}} W_{\Delta_2}$. In the block decomposition corresponding to this splitting, two geodesic rays might go through the same sequence of hyperbolic planes corresponding to $\W{\Delta_1}$ and $\W{\Delta_2}$ but their rays are not asymptotic. In other words, there might be pairs of distinct points in the visual boundary having the same infinite itinerary. We have proven in \cref{Morse_conn_of_type_Ainf} that such an unpleasant situation does not occur among Morse geodesic rays of infinite itinerary. 
\end{remark}

For completing this section, we show that the graph $\Delta$ in the left upper corner in \cref{fig:exCFS} (mentioned in the introduction, there \cref{fig}) is contained in $\mathcal C$. For the proof, we consider the successive decomposition pictured in \cref{fig:exCFS}.

 	\begin{figure}[h]
	\centering
	\includegraphics{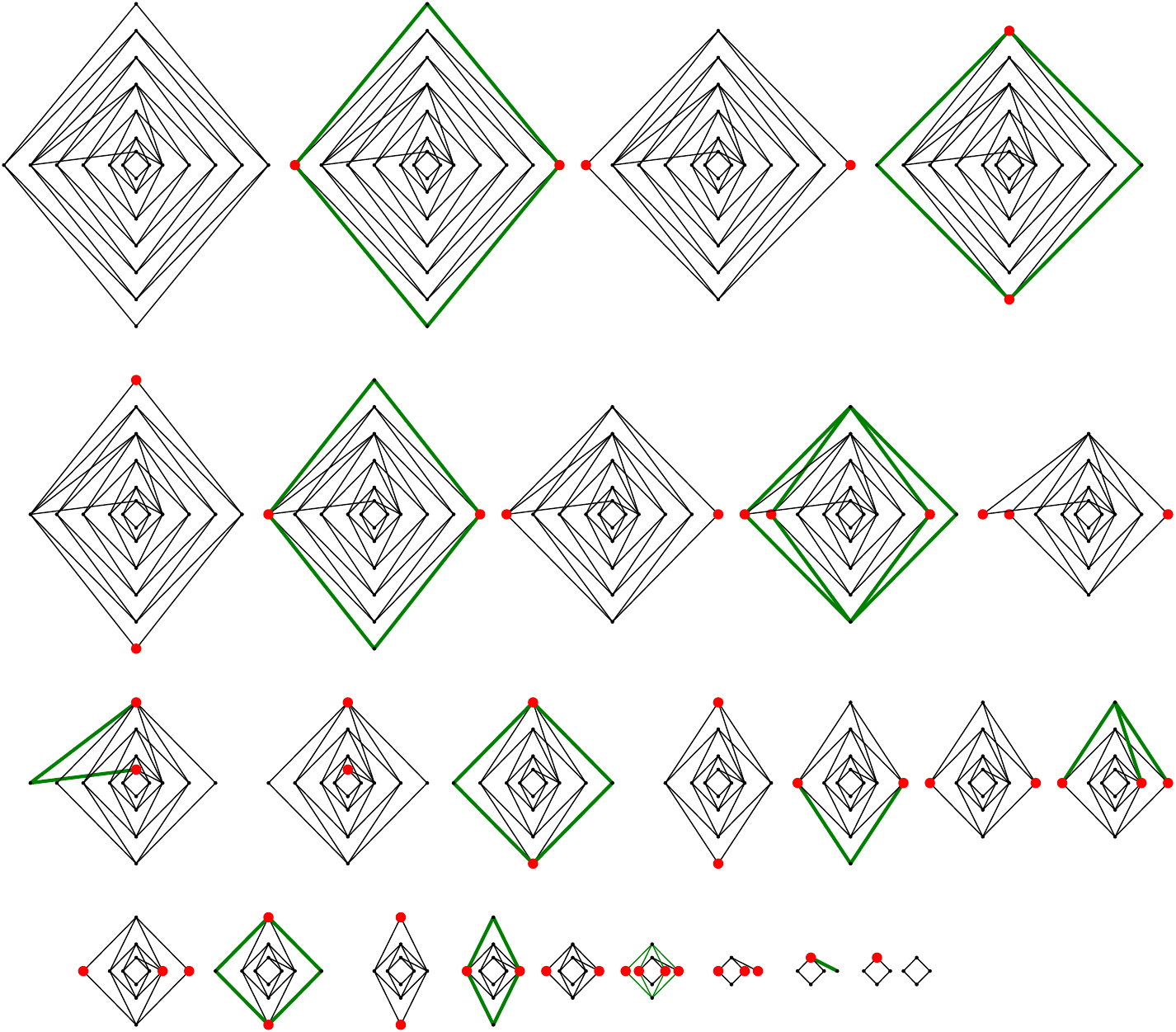}	
	\caption{Decomposition of the graph $\Delta$ in  the left upper corner that was studied in \cite[Example 7.7]{Tran_hierarchically} (see \cref{fig}).  The Decomposition shows that $\Delta$ is contained in $\mathcal C$. }
	\label{fig:exCFS}
\end{figure}
The graph $\Graph$ is seen in the left upper corner. We decompose $\Graph$ from left to right and from above to bottom. 
In each second step, we decompose the graph into a green and a black graph. The intersection of these two graphs always consists of single vertices marked by the thick red points. These red vertices are contained either in a non-trivial join or in a clique. In every second step, we delete the green subgraph and continue to decompose the resulting graph in the next step. Finally, we end up with a $4$-cycle. As a $4$-cycle is contained in $\mathcal C$, we conclude that $\Graph \in \mathcal C$.
\begin{remark}
 The graph $\Delta$ in the left upper corner in \cref{fig:exCFS} is not planar. Thus, it is not contained in $\mathcal{CFS}_0$.
\end{remark}
 \begin{remark}
Behrstock \cite{Behrstock} investigates a graph $\Delta'$ similar to $\Delta$ and shows that $\morse \Sigma_{\Delta'}$ contains a circle. 
		This circle corresponds to an induced $5$-cycle $C$ in $\Delta'$. No pair of non-adjacent vertices of this cycle $C$ is contained in an induced $4$-cycle. In particular, no induced subgraph of $C$ is contained in a non-trivial join (since in a non-trivial join, any pair of non-adjacent vertices is contained in an induced 4-cycle). Hence, no matter how often and in which way we decompose the graph $\Delta'$ along graphs that are contained in non-trivial joins or cliques, the remaining graph will always contain $C$. Interestingly, it is possible to decompose $\Delta'$ similarly to $\Delta$ so that only the $5$-cycle $C$ remains. 
\end{remark}

\section{Beyond RACGs}
 \label{examoplesection}
In this section, we study applications of \cref{thm2} that are not RACGs.
\subsection{Right-angled Artin groups (RAAG)}

The \textit{right-angled Artin group (RAAG)} associated to a finite, simplicial graph $\Delta=(V,E)$ is the group 
\[A_\Delta =\langle V\mid uv=vu ~\forall ~\{u,v\} \in E \rangle.\]
The group $A_\Delta$ acts geometrically on an associated CAT(0) cube complex $\Sigma^A_{\Delta}$, its \textit{Salvetti complex}. Hence, the Morse boundary of $A_\Delta$ is the Morse boundary of $\Sigma^A_{\Delta}$.
\begin{example}[Croke-Kleiner-space~\cite{CrokeKleiner}]
		 	\label{exCrokeKleiner}
		 	\begin{figure}[h]
		\centering
		\includegraphics{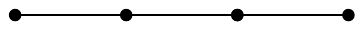}	
		\caption{The RAAG studied by Croke an Kleiner \cite{CrokeKleiner}.}
		\label{fig:CrokeKleiner}
	\end{figure}

Croke--Kleiner study a right-angled Artin group (RAAG) $A_\Delta$ whose defining graph is pictured in \cref{fig:CrokeKleiner}. This group admits a splitting of the form \[A_\Delta = (F_2 \times \Z )*_{\Z^2}(F_2 \times \Z ).\]
This splitting corresponds to a treelike block decomposition of the Salvetti complex $\Sigma^A_{\Delta}$ on which $A_\Delta$ acts geometrically. The walls in this block decomposition are Euclidean flats. Thus, we can apply \cref{thm2}. The Morse boundary of $F_2 \times \Z$ is empty as $F_2 \times \Z$ is a direct product of two infinite CAT(0) groups. We apply \cref{thm2} and conclude that $\morse A_\Delta$ is totally disconnected. Though the factors in the splitting of $A_\Delta$ have empty Morse boundary, the Morse boundary $\morse A_\Delta$ is not empty. It consists of Morse geodesic rays that don't end in a block. Accordingly, each connected component of $\morse A_\Delta$ is of type \typeAinf, i.e. each connected component consists of an equivalence class of a geodesic ray with infinite itinerary. 
\end{example}

More generally, we can transfer the line of argumentation in \cref{sec:mainproofs} to RAAGs by studying Salvetti-complexes of RAAGs instead of Davis complexes of RACGs.
We conclude similarly to the case of RACGs: If $\Delta$ is a join of two graphs, then $A_\Delta$ is a direct product of two RAAGs. As each RAAG is an infinite group, each such RAAG is a direct product of two infinite CAT(0) groups. In such a case, each geodesic ray in $\Sigma_\Delta^A$ is bounded by a Euclidean half-plane, and $A_\Delta$ has empty Morse boundary. Corollary B in \cite{Sageev_Caprace} implies the following lemma:
\begin{lemma}[Sageev--Caprace]
	\label{emptyRaags}
	The Morse boundary of a RAAG $A_\Delta$ is empty if and only if $\Delta$ is the join of two non-empty graphs. 
\end{lemma}

Now, let $\Delta$ be a finite, simplicial graph that can be decomposed into two distinct proper induced subgraphs $\Delta_1$ and $\Delta_2$ with the intersection graph $\Lambda=\Delta_1 \cap \Delta_2$. Repeating the arguments in the proof of \cref{Prop} in the setting of RAAGs yields 

\begin{prop}
	\label{RAAGprop}	
The collection $\mathcal B \coloneqq \{g\Sigma^A_{\graphzero}\mid g \in A_{\Graph}\}\cup \{g\Sigma^A_{\graphone}\mid g \in A_{\Graph}\}$
is a treelike block decomposition of $\Sigma^A_{\Delta}$. 
The collection of walls $\mathcal W$ is given by $\mathcal W = \{g\Sigma^A_{\Lambda}\mid g \in A_{\Graph}\}$.
\end{prop}

\cref{thm2} combined with \cref{RAAGprop} and \cref{emptyRaags} directly implies 

 \begin{thm}
	\label{thm1RAG}
	Suppose that $\Lambda$ is contained in a join of two induced subgraphs of $\Graph$.
	Then every connected component of $\morse \Sigma^A_\Delta$ is either
	
	\begin{enumerate}
		\item a single point; or
		\item homeomorphic to a connected component of $\relMorse{\Sigma^A_{\Delta}}{\Sigma^A_{\Delta_i}}$ equipped with the subspace topology of  $\morse \Sigma^A_\Delta$ where $i \in \{1,2\}$. 
	\end{enumerate} 	
\end{thm}
By means of \cref{corOpenset} we obtain
 \begin{cor}
	\label{Cor1Artin}
	Suppose that the assumptions of \cref{thm1RAG} are satisfied. 
	If $\relMorse{\Sigma^A_{\Delta}}{\Sigma_{\Delta_1}}$ and $\relMorse{\Sigma^A_{\Delta}}{\Sigma_{\Delta_2}}$ equipped with the subspace topology of $\morse \Sigma^A_{\Delta_1}$ and $\morse \Sigma^A_{\Delta_2}$ are totally disconnected then 
	$\morse \Sigma^A_\Delta$ is totally disconnected. 
\end{cor}
We need the following lemma for studying Charney-Sultan graphs in the setting of RAAGs.

\begin{lemma}
	If $\Delta$ is a finite tree, then $A_{\Delta}$ has totally disconnected Morse boundary.
	\label{Artintree}
\end{lemma}
\begin{proof}
	Suppose that $\Delta$ is a finite tree. We show by induction on the numbers of edges in $\Delta$ that $A_\Delta$ has totally disconnected Morse boundary. If $\Delta$ consists of an edge, $A_\Delta$ is isomorphic to $\Z^2$ and $\morse A_\Delta$ is empty. Now suppose that $\Delta$ is a tree with $n$ edges. Then there exists an edge $e$ such that $\Delta$ is obtained by gluing one endvertex of $e$ to a subree $T$ of $\Delta$ which contains $n-1$ vertices. By the induction hypothesis, $A_T$ has totally disconnected Morse boundary. As $A_e$ has empty Morse boundary, it follows from \cref{Cor1Artin} that $A_\Delta$ has totally disconnected Morse boundary.\end{proof}
Adapting the arguments in \cref{lem:charsultgraph} by dint of \cref{Artintree} yields
\begin{lemma}
	\label{lem:CSAAG}
		If $\Delta$ is a Charney-Sultan graph, then $A_{\Delta}$ has totally disconnected Morse boundary.
\end{lemma}

Let $\mathcal C'$ be the graph class obtained by varying the \cref{def:graphclass} in the following manner: 
\begin{itemize}
	\item replace $\mathcal C$ by $\mathcal C'$;
	\item replace the condition \itemref{def:graphclass}{itemgraphclass2.2} in \cref{def:graphclass} by the condition that each (not necessarily non-trivial) join of two non-empty graphs is contained in $\mathcal C'$. 
\end{itemize}
Applying \cref{emptyRaags}, \cref{Cor1Artin}, \cref{Artintree} and  \cref{lem:CSAAG} similarly to the argumentation after \cref{def:graphclass}   leads to the following corollary which gives an alternative prove of \cref{Artingroups} for RAAGs whose defining graphs are contained in $\mathcal C'$. 


\begin{cor}\label{RAAGThm}
	If $\Delta \in \mathcal C'$, then $A_\Delta$ has totally disconnected Morse boundary. 
\end{cor}	

\newpage
\subsection{Surface amalgams}
In this section we study examples of surface amalgams that were examined by Ben-Zvi in \cite{BenzviFlats}.	\begin{figure}[h]
	\centering
	\includegraphics{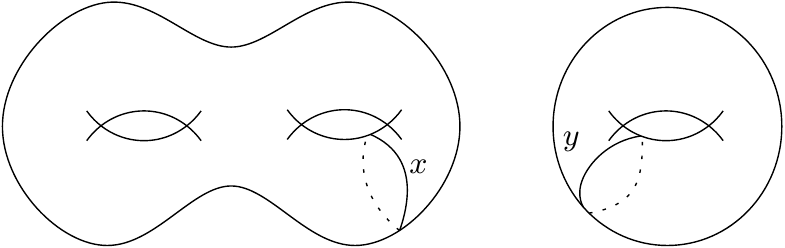}	
	\caption{Example 2.1 in \cite{BenzviFlats}: A torus and a genus 2 surface identified along the curves $x$ and $y$.}
	\label{fig:exBenZvi2.1}
\end{figure} 
\begin{example}[Example 2.1 in \cite{BenzviFlats}]
		
		Let $G_1$ be the fundamental group of the surface amalgam pictured in \cref{fig:exBenZvi2.1}. 
		Its universal cover $\Sigma_1$ admits a treelike block decomposition in blocks that are Euclidean and hyperbolic planes corresponding to the two-torus $T_2$ on the left and the torus $T_1$ on the right. As in the example of Charney--Sultan pictured in \cref{fig:exCharneySultan}, one can argue that the relative Morse boundary $\relMorse{\Sigma_1}{\tilde T_2}$ of the universal cover of the two-torus $\tilde T_2$ on the left endowed with the subspace topology $\morse \tilde T_2$ is totally disconnected. The relative Morse boundary $\relMorse{\Sigma_1}{\tilde T_1}$ of the universal cover of the torus $\tilde T_1$ on the right is empty. By \cref{cor:firstgeneralization}, the Morse boundary of $G_1$ is totally disconnected. 
	\begin{remark}
		Ben-Zvi shows that the group $G_1$ is a CAT(0) group with isolated flats and that $\vis \Sigma_1$ is path-connected. The isolated flat property implies that two rays passing through the same infinite collection of hyperbolic or Euclidean planes are asymptotic. So, 
		in this example, \cref{Morse_conn_of_type_Ainf} is also true for non-Morse geodesic rays.
	\end{remark}
\end{example}

\begin{example}[Example 2.2 in \cite{BenzviFlats}]
Let $G_2$ be the fundamental group of two tori with boundary components identified as shown in \cref{fig:exBenZviSchwer}.
\begin{figure}[h]
	\centering
	\includegraphics{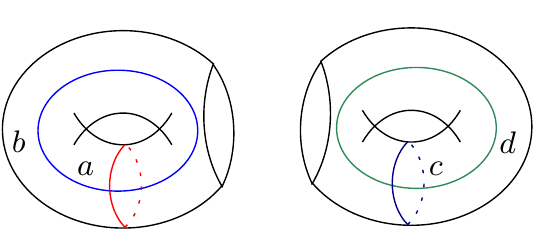}	
	\caption{Example 2.2 in \cite{BenzviFlats}: Two tori with boundary components that we identify via $[a,b]^2=[c,d]^2$. }
	\label{fig:exBenZviSchwer}
\end{figure}
 By applying \cref{thm2} twice, we prove that $\morse G_2$ is totally disconnected:
\begin{figure}[h]
	\centering
	\includegraphics{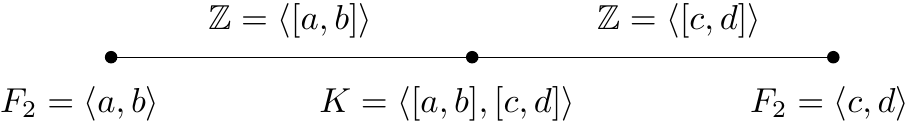}	
	\caption{Example 2.2 in \cite{BenzviFlats}: The maximal peripheral splitting of $G_2$. }
	\label{fig:exBenZviSchwersplitting}
\end{figure}
The identification of $[a,b]^2$ and $[c,b]^2$ creates a Klein bottle group $K$ generated by $[a,b]$ and $[c,d]$ and 
$G_2$ admits the splitting as in \cref{fig:exBenZviSchwersplitting}. Like Ben-Zvi \cite{BenzviFlats}, we use the equivariant gluing theorem of Bridson--Haefliger~\cite[Thm. II.11.18]{BH} for constructing a CAT(0) space $\Sigma_2$ on which $G_2$ acts geometrically: 
First, we take Euclidean planes for the Klein bottle group and trunked hyperbolic planes for the free groups generated by $a$ and $b$ and glue them together as in the equivariant gluing theorem of Bridson--Haefliger. This way, we obtain a space $\Sigma_2'$ on which the group $G' = \langle a,b \rangle \times_{\langle [a,b]\rangle} \times K$ acts geometrically. 
This space is a CAT(0) space with a treelike block decomposition. Its blocks consist of Euclidean planes corresponding to the Klein bottle group and trunked hyperbolic planes corresponding to the free groups. As the Klein bottle group is virtually $\Z^2$, its Morse boundary is empty and no wall of $\Sigma_2'$ contains a Morse geodesic ray. 
Since the Morse boundary of $F_2$ is a Cantor set and the Morse boundary of $K$ is empty, \cref{cor:firstgeneralization} implies that the Morse boundary of $\Sigma_2'$ is totally disconnected. 

Now, we construct the space $\Sigma_2$ on which the group $G_2$ acts geometrically. For that purpose, we glue copies of $\Sigma_2'$ along bi-infinite geodesic rays corresponding to $\Z = \langle [c,d]\rangle$ to copies of trunked hyperbolic planes corresponding to the free group $F_2=\langle c,d \rangle$ as in the equivariant gluing theorem of Bridson--Haefliger. This way, we obtain a CAT(0) space with a treelike block decomposition where each wall is a bi-infinite geodesic ray that is contained in a subspace corresponding to the Kleinian group $K$. Hence, no wall contains a Morse geodesic ray and the Morse boundary of $G_2$ is totally disconnected
by \cref{cor:firstgeneralization}.
\begin{remark}
	 Ben-Zvi shows that $G_2$ is a CAT(0) group with isolated flats and that the visual boundary of $\Sigma_2$ is path-connected.
\end{remark}
\end{example}

\begin{example}[Examples arising from the equivariant gluing theorem of Bridson--Haefliger]
	 The spaces arising from the equivariant gluing theorem \cite[Theorem II.11.18]{BH} of Bridson--Haefliger are CAT(0) spaces with a treelike block decompositions on which amalgamated free products of CAT(0) groups act geometrically as observed by Ben-Zvi in Example 6.8 in \cite{BenzviFlats}. Thus, many other examples can be constructed to which \cref{thm2} and \cref{cor:firstgeneralization} can be applied.
\end{example}


\bibliographystyle{alpha}
\bibliography{Bibliography}

\begin{thebibliography}{BFRHS18}

\bibitem[Beh19]{Behrstock}
Jason Behrstock.
\newblock A counterexample to questions about boundaries, stability, and
  commensurability.
\newblock In {\em Beyond hyperbolicity}, volume 454 of {\em London Math. Soc.
  Lecture Note Ser.}, pages 151--159. Cambridge Univ. Press, Cambridge, 2019.

\bibitem[BFRHS18]{Behrstock_random}
Jason Behrstock, Victor Falgas-Ravry, Mark~F. Hagen, and Tim Susse.
\newblock Global structural properties of random graphs.
\newblock {\em Int. Math. Res. Not. IMRN}, (5):1411--1441, 2018.

\bibitem[BH99]{BH}
Martin~R. Bridson and Andr\'{e} Haefliger.
\newblock {\em Metric spaces of non-positive curvature}, volume 319 of {\em
  Grundlehren der Mathematischen Wissenschaften [Fundamental Principles of
  Mathematical Sciences]}.
\newblock Springer-Verlag, Berlin, 1999.

\bibitem[BZ]{BenzviFlats}
Michael Ben-Zvi.
\newblock Boundaries of groups with isolated flats are path connected.
\newblock arXiv:1909.12360, 2019.

\bibitem[BZK21]{BENZVI}
Michael Ben-Zvi and Robert Kropholler.
\newblock Right-angled {A}rtin group boundaries.
\newblock {\em Proc. Amer. Math. Soc.}, 149(2):555--567, 2021.

\bibitem[Cas16]{Cashen}
Christopher~H. Cashen.
\newblock Quasi-isometries need not induce homeomorphisms of contracting
  boundaries with the {G}romov product topology.
\newblock {\em Anal. Geom. Metr. Spaces}, 4(1):278--281, 2016.

\bibitem[CCS]{Artingroups}
Ruth Charney, Matthew Cordes, and Alessandro Sisto.
\newblock Complete topological descriptions of certain morse boundaries.
\newblock arXiv:1908.03542, 2019. To appear in Groups Geom. Dyn..

\bibitem[CK00]{CrokeKleiner}
Christopher~B. Croke and Bruce Kleiner.
\newblock Spaces with nonpositive curvature and their ideal boundaries.
\newblock {\em Topology}, 39(3):549--556, 2000.

\bibitem[Cor17]{Cordes_properMorse}
Matthew Cordes.
\newblock Morse boundaries of proper geodesic metric spaces.
\newblock {\em Groups Geom. Dyn.}, 11(4):1281--1306, 2017.

\bibitem[CS11]{Sageev_Caprace}
Pierre-Emmanuel Caprace and Michah Sageev.
\newblock Rank rigidity for {CAT}(0) cube complexes.
\newblock {\em Geom. Funct. Anal.}, 21(4):851--891, 2011.

\bibitem[CS15]{CharSul}
Ruth Charney and Harold Sultan.
\newblock Contracting boundaries of {$\mathrm{CAT}(0)$} spaces.
\newblock {\em J. Topol.}, 8(1):93--117, 2015.

\bibitem[Dav08]{Davis}
Michael~W. Davis.
\newblock {\em The geometry and topology of {C}oxeter groups}, volume~32 of
  {\em London Mathematical Society Monographs Series}.
\newblock Princeton University Press, Princeton, NJ, 2008.

\bibitem[DGO17]{DahmaniandCo17}
F.~Dahmani, V.~Guirardel, and D.~Osin.
\newblock Hyperbolically embedded subgroups and rotating families in groups
  acting on hyperbolic spaces.
\newblock {\em Mem. Amer. Math. Soc.}, 245(1156):v+152, 2017.

\bibitem[DT15]{Dani_diver}
Pallavi Dani and Anne Thomas.
\newblock Divergence in right-angled {C}oxeter groups.
\newblock {\em Trans. Amer. Math. Soc.}, 367(5):3549--3577, 2015.

\bibitem[Gen20]{Genevois}
Anthony Genevois.
\newblock Hyperbolicities in {CAT}(0) cube complexes.
\newblock {\em Enseign. Math.}, 65(1-2):33--100, 2020.

\bibitem[GKLS21]{graeber2020surprising}
Marius Graeber, Annette Karrer, Nir Lazarovich, and Emily Stark.
\newblock Surprising circles in morse boundaries of right-angled coxeter
  groups.
\newblock {\em Topol. Appl.}, 294:107645, 2021.

\bibitem[Gro87]{Gromov}
M.~Gromov.
\newblock Hyperbolic groups.
\newblock In {\em Essays in group theory}, volume~8 of {\em Math. Sci. Res.
  Inst. Publ.}, pages 75--263. Springer, New York, 1987.

\bibitem[HW08]{Haglund2008}
Fr{\'e}d{\'e}ric Haglund and Daniel~T. Wise.
\newblock Special cube complexes.
\newblock {\em Geom. Funct. Anal.}, 17(5):1551--1620, 2008.

\bibitem[Lev18]{Ivan}
Ivan Levcovitz.
\newblock Divergence of {$\mathrm CAT(0)$} cube complexes and {C}oxeter groups.
\newblock {\em Algebr. Geom. Topol.}, 18(3):1633--1673, 2018.

\bibitem[Moo10]{Mooney}
Christopher Mooney.
\newblock Generalizing the {C}roke-{K}leiner {C}onstruction.
\newblock {\em Topol. Appl.}, 157:1168--1181, 2010.

\bibitem[Mun00]{Topology}
James~R. Munkres.
\newblock {\em Topology}.
\newblock Prentice Hall, Inc., Upper Saddle River, NJ, 2000.

\bibitem[NT19]{Tran_coars_geom}
Hoang~Thanh Nguyen and Hung~Cong Tran.
\newblock On the coarse geometry of certain right-angled {C}oxeter groups.
\newblock {\em Algebr. Geom. Topol.}, 19(6):3075--3118, 2019.

\bibitem[Osi16]{Osin16}
D.~Osin.
\newblock Acylindrically hyperbolic groups.
\newblock {\em Trans. Amer. Math. Soc.}, 368(2):851--888, 2016.

\bibitem[RST]{Tran_hierarchically}
Jacob Russell, Davide Spriano, and Hung~Cong Tran.
\newblock Convexity in hierarchically hyperbolic spaces.
\newblock arXiv:1809.09303, 2018.

\bibitem[Sag95]{Sageev}
Michah Sageev.
\newblock Ends of group pairs and non-positively curved cube complexes.
\newblock {\em Proc. London Math. Soc. (3)}, 71(3):585--617, 1995.

\bibitem[Sag14]{SageevNotes}
Michah Sageev.
\newblock {$\mathrm CAT(0)$} cube complexes and groups.
\newblock In {\em Geometric group theory}, volume~21 of {\em IAS/Park City
  Math. Ser.}, pages 7--54. Amer. Math. Soc., Providence, RI, 2014.

\bibitem[Sis16]{Sisto16}
Alessandro Sisto.
\newblock Quasi-convexity of hyperbolically embedded subgroups.
\newblock {\em Math. Z.}, 283(3-4):649--658, 2016.

\bibitem[Sis18]{Sisto18}
Alessandro Sisto.
\newblock Contracting elements and random walks.
\newblock {\em J. Reine Angew. Math.}, 742:79--114, 2018.

\bibitem[Sul14]{Sultan}
Harold Sultan.
\newblock Hyperbolic quasi-geodesics in {CAT}(0) spaces.
\newblock {\em Geom. Dedicata}, 169:209--224, 2014.

\bibitem[Tra19]{Tran}
Hung~Cong Tran.
\newblock On strongly quasiconvex subgroups.
\newblock {\em Geom. Topol.}, 23(3):1173--1235, 2019.

\bibitem[Wes01]{West}
Douglas~B. West.
\newblock {\em Introduction to graph theory}.
\newblock Prentice Hall, Upper Saddle River, NJ, 2. ed. edition, 2001.

\end{thebibliography}


\end{document}